\documentclass[11pt]{article}

\usepackage[margin=3.5cm]{geometry}

\usepackage{color}
\usepackage[round]{natbib}
\usepackage{mathrsfs}
\usepackage{stmaryrd}
\usepackage{multicol}
\usepackage{tikz, mathtools}

\usetikzlibrary{arrows}
\usepackage{graphicx}

\usepackage{amsmath, amsfonts, amsthm, amssymb, bbm, bm}
\usepackage{mathrsfs}
\usepackage{enumerate}
\usepackage{tabularx}
\usepackage{multicol}
\usepackage{multirow}

\newcommand{\reals}{\mathbb{R}}

\newcommand{\G}{\mathcal{G}}
\newcommand{\X}{\mathfrak{X}}

\DeclareMathOperator{\pa}{pa}
\DeclareMathOperator{\ch}{ch}

\DeclareMathOperator{\an}{an}

\DeclareMathOperator{\sterile}{sterile}
\DeclareMathOperator{\dis}{dis}

\DeclareMathOperator{\pre}{pre}

\newcommand\indep{\protect\mathpalette{\protect\independenT}{\perp}}
\def\independenT#1#2{\mathrel{\rlap{$#1#2$}\mkern2mu{#1#2}}}

\theoremstyle{plain}
\newtheorem{lem}{Lemma}[section]
\newtheorem{thm}[lem]{Theorem}
\newtheorem{prop}[lem]{Proposition}
\newtheorem{cor}[lem]{Corollary}

\theoremstyle{definition}
\newtheorem{dfn}[lem]{Definition}
\newtheorem{rmk}[lem]{Remark}
\newtheorem{exm}[lem]{Example}

\newcommand{\benum}{\begin{enumerate}}
\newcommand{\eenum}{\end{enumerate}}

\newcommand{\bitem}{\begin{itemize}}
\newcommand{\eitem}{\end{itemize}}

\newcommand{\barr}{\begin{array}}
\newcommand{\earr}{\end{array}}

\newcommand{\bmat}{\begin{pmatrix}}
\newcommand{\emat}{\end{pmatrix}}

\newcommand{\blist}{\renewcommand{\labelenumi}{\textbf{\arabic{enumi}}.} \begin{enumerate}}
\newcommand{\elist}{\end{enumerate} \renewcommand{\labelenumi}{\arabic{enumi}.}}

\def\bal#1\eal{\begin{align*}#1\end{align*}}

\newcommand{\partn}[1]{\llbracket #1 \rrbracket}

\tikzstyle{rv}=[circle, draw, thick, minimum size=7mm, inner sep=.5mm]
\tikzstyle{fv}=[rectangle, draw, thick, minimum size=7mm, inner sep=.5mm]
\tikzstyle{da}=[->, very thick, color=blue]
\tikzstyle{ba}=[<->, very thick, color=red]

\newcommand{\RF}{\mathcal{M}_{\textit{rf}}}
\newcommand{\PA}{\mathcal{M}_p}
\newcommand{\fd}{\mathfrak{d}}
\newcommand{\fm}{\mathfrak{m}}

\DeclareMathOperator{\rh}{rh}
\DeclareMathOperator{\Do}{do}

\title{Smooth, identifiable supermodels of\\
discrete DAG models with latent variables}

\author{ {\bf Robin J.~Evans
} \\  Department of Statistics \\  University of Oxford\\ \tt{evans@stats.ox.ac.uk}\\
\and {\bf Thomas S.~Richardson}  \\ Department of Statistics \\ University of Washington\\
\tt{thomasr@u.washington.edu}\\
}

\begin{document}

\maketitle

\begin{abstract}
  We provide a parameterization of the discrete nested Markov model,
  which is a supermodel that approximates DAG models (Bayesian network 
  models) with latent variables.  Such models are widely used in causal
  inference and machine learning.  We explicitly evaluate their
  dimension, show that they are curved exponential
  families of distributions, and fit them to data.  The
  parameterization avoids the irregularities and unidentifiability of
  latent variable models.  The parameters used are all fully
  identifiable and causally-interpretable quantities.
\end{abstract}

\section{Introduction}

Directed acyclic graph (DAG) models, also known as Bayesian networks,
are a widely used class of multivariate models in probabilistic
reasoning, machine learning and causal inference \citep{bishop:07,
  darwiche:09, pearl:09}.  The inclusion of latent variables within
Bayesian network models can greatly increase their flexibility, and
also account for unobserved confounding; however, latent variable
models are typically non-regular, their dimension can be hard to calculate, 
and they generally do not have fully 
identifiable parameterizations.  In this paper we will present an
alternative approach which overcomes these difficulties, and does not
require any parametric assumptions to be made about the latent
variables.

\begin{exm} \label{exm:verma} Suppose we are interested in the
  relationship between family income during childhood $X$, an individual's 
  education level $E$, their military service $M$, and 
  their later income $Y$.  We might
  propose the model shown in Figure \ref{fig:verma}(a), which includes
  a hidden variable $U$ representing motivation or intelligence.  Let
  the four observed variables be binary, but make no assumption
  about $U$.

  One can check using Pearl's d-separation criterion \citep{pearl:09}
  that $M \indep X \,|\, E$ under this model, i.e.\ there is no
  relationship between military service and family income after
  controlling for level of education; this places two independent
  constraints on the variables' joint distribution $p(x, e, m, y)$ (one for each level
  of $E$).  In addition, let
  $q_{EY}(e, y \,|\, x, m) \equiv p(e \,|\, x) \cdot p(y \,|\, x, m,
  e)$; then the quantity
  \begin{align}
q_{EY}(y \,|\, x, m) &\equiv \sum_e q_{EY}(e, y \,|\, x, m) \nonumber \\
&= \sum_{e} p(e \,|\, x) \cdot p(y \,|\, x, m, e) \label{eqn:cause}
\end{align}
does not depend upon $x$ \citep{robins:86}; a short proof of this
is given in Appendix \ref{sec:dorm}.  
If the graph is
interpreted causally then
$q_{EY}(y \,|\, x, m) = p(y \,|\, \Do(x, m))$, i.e.\ it is the
distribution of $Y$ in an experiment that externally sets
$\{X=x, M = m\}$.  Note that generally
$q_{EY}(y \,|\, x, m) \neq p(y \,|\, x, m)$.

\begin{figure}
\begin{center}
\begin{tikzpicture}
[>=stealth, node distance=15mm]
\pgfsetarrows{latex-latex};
\begin{scope}
 \node[rv] (1) {$X$};
 \node[rv, right of=1] (2) {$E$};
 \node[rv, right of=2] (3) {$M$};
 \node[rv, right of=3] (4) {$Y$};
 \draw[da] (1) -- (2);
 \draw[da] (2) -- (3);
 \draw[da] (3) -- (4);
\draw[ba] (2.40) .. controls +(40:1) and +(140:1) .. (4.140);
\node[below of=2, xshift=7.5mm, yshift=2mm] {(b)};
\end{scope}
\begin{scope}[xshift=-7cm]
 \node[rv] (1) {$X$};
 \node[rv, right of=1] (2) {$E$};
 \node[rv, right of=2] (3) {$M$};
 \node[rv, right of=3] (4) {$Y$};
 \node[rv, above of=3, yshift=-5mm, color=red] (5) {$U$};
 \draw[da] (1) -- (2);
 \draw[da] (2) -- (3);
 \draw[da] (3) -- (4);
 \draw[da, color=red] (5) -- (4);
 \draw[da, color=red] (5) -- (2);
\node[below of=2, xshift=7.5mm, yshift=2mm] {(a)};
\end{scope}
\end{tikzpicture}
\caption{(a) A directed acyclic graph with the latent variable $U$; (b)
  a (conditional) acyclic directed mixed graph (the \emph{Verma
    graph}) representing the observed distribution in (a).}
\label{fig:verma}
\end{center}
\end{figure}
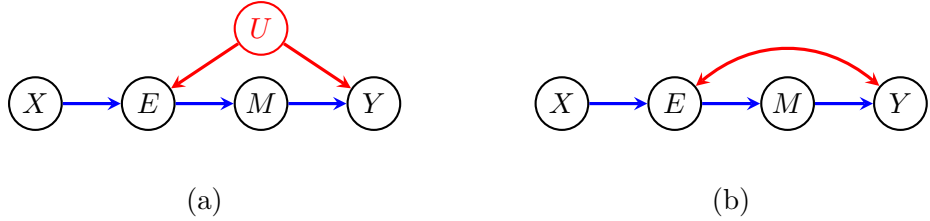

The restriction that (\ref{eqn:cause}) does not depend on $x$ corresponds to two further independent
constraints on $p$, one for each level of $m$.  The set of
distributions that satisfy all four constraints is the \emph{nested
  Markov model} associated with the graph(s) in Figure
\ref{fig:verma}; the number of free parameters is $15 - 2 - 2 = 11$.

The distributions in the model all factorize as
\begin{align*}
p(x, e, m, y) &= p(x) \cdot p(m \,|\, e) \cdot q_{EY}(e, y \,|\, x, m),
\end{align*}
and each of the three factors can be parameterized separately.  The
model can therefore be described using the following 11 free
parameters:
\begin{align*}
& p(x = 0),  && p(m = 0 \,|\, e),
&&  q_{EY}(e = 0 \,|\, x),\\
&&& q_{EY}(y = 0 \,|\, m), && q_{EY}(e = y = 0 \,|\, x, m). 
\end{align*}
which---if we interpret the model as a causal one---are respectively
the quantities
\begin{align*}
& P(X = 0),  && P(M = 0 \,|\, \Do(E=e)),
&&  P(E = 0 \,|\, \Do(X=x)),\\
&&& P(Y = 0 \,|\, \Do(M=m)), && P(E=0, Y=0 \,|\, \Do(X=x, M=m)). 
\end{align*}
The map from the set of positive probability distributions that satisfy the 4
constraints to these 11 parameters is smooth and bijective, and the
parameters are fully identifiable.  It follows that the model is a
curved exponential family of distributions, and that it can be fitted
using standard numerical methods.
\end{exm}

An alternative modelling approach would be to include a latent variable $U$
explicitly in the model, but this leads to some parameters being
unidentifiable.  For example, with a binary $U$ the model implied by
Figure \ref{fig:verma}(a) has 12 parameters.  We know that the true
marginal distribution has at most dimension 11, so at least one of
these 12 parameters is unidentifiable.  Even though the model is not identified, 
this latent variable model is still `too small', in the sense 
that the model over the observed margin only has dimension 10, whereas
dimension 11 can be obtained if $U$ is allowed to have more than two states.  
As $U$ is not observed, it is undesirable to make specific assumptions about $U$'s
state-space because one may unwittingly impose restrictions on the observable 
distribution.  Further, latent variable models are not statistically
regular, so standard statistical theory for likelihood ratio tests and
asymptotic normality of parameter estimates does not apply
\citep[see, e.g.][]{mond:03, drton:09a}.  

\subsection{Other Work and this Paper's Contribution}

Models of conditional independence associated with margins of 
DAG models (we refer to these as `ordinary Markov models') 
have been studied by \citet{richardson:03}; see also \citet{wermuth:11}.  These
models were parameterized and shown to be smooth by 
\citet{evans:14}.  
Other approaches using probit models \citep{silva:09} and cumulative
distribution networks \citep{huang:08, silva:11} are more parsimonious 
than ordinary Markov models, but impose additional constraints due to 
their parametric structure.

None of the models mentioned in the previous paragraph can account
for constraints of the kind in (\ref{eqn:cause}), which were first identified by
\citet{robins:86} and separately by \citet{verma:91}.  Such constraints
are attractive because they allow finer distinctions between different 
causal models from purely observational data: for example, going by conditional
independence alone the graph in Figure \ref{fig:verma}(b) is Markov equivalent 
to the DAGs in Figure \ref{fig:verma_eq}, and these causal models are therefore 
observationally indistinguishable; however the DAGs do not imply the Verma
constraint (\ref{eqn:cause}), so under the nested Markov model one \emph{can} 
distinguish between these models.

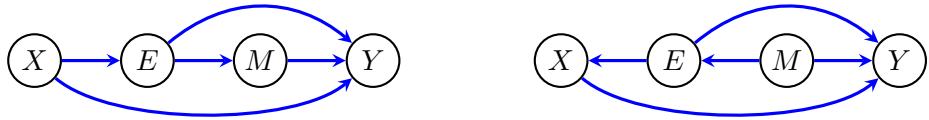
\begin{figure}
\begin{center}
\begin{tikzpicture}
[>=stealth, node distance=15mm]
\pgfsetarrows{latex-latex};
\begin{scope}
 \node[rv] (1) {$X$};
 \node[rv, right of=1] (2) {$E$};
 \node[rv, right of=2] (3) {$M$};
 \node[rv, right of=3] (4) {$Y$};
 \draw[da] (1) -- (2);
 \draw[da] (2) -- (3);
 \draw[da] (3) -- (4);
\draw[da] (2.40) .. controls +(40:1) and +(140:1) .. (4.140);
\draw[da] (1.320) .. controls +(320:1) and +(220:1) .. (4.220);
\end{scope}
\begin{scope}[xshift=7cm]
 \node[rv] (1) {$X$};
 \node[rv, right of=1] (2) {$E$};
 \node[rv, right of=2] (3) {$M$};
 \node[rv, right of=3] (4) {$Y$};
 \draw[da] (2) -- (1);
 \draw[da] (3) -- (2);
 \draw[da] (3) -- (4);
\draw[da] (2.40) .. controls +(40:1) and +(140:1) .. (4.140);
\draw[da] (1.320) .. controls +(320:1) and +(220:1) .. (4.220);
\end{scope}
\end{tikzpicture}
\caption{DAGs that represent the same conditional 
independence model as Figure \ref{fig:verma}(b), but 
which do not imply the Verma constraint.}
\label{fig:verma_eq}
\end{center}
\end{figure}

An algorithm
for finding such constraints was given by \citet{tian:02a}, and 
developed into a fully nonparametric statistical model (the nested Markov 
model) by \citet{shpitser:17}.  
In this paper we provide a smooth, statistically regular 
and fully identifiable parameterization of the discrete version 
of nested Markov models.  As a result, discrete nested Markov
models are shown to be curved exponential families of distributions of
known dimension.  All the parameters we derive are interpretable as
straightforward causal quantities.
\citet{evans:complete} shows that the discrete nested Markov model 
that we describe here is the best possible algebraic approximation 
to DAG models with latent
variables, in the sense that the models have the same dimension over
the observed variables.  An earlier review paper \citep{shpitser:14}
mentions the parameterization given here, but no proofs are provided. 

The conditional independence constraints we consider here 
include the constraints described in \citep{tian:02a}. They are also
a special case of the \emph{dormant independences}.
However, not all dormant independences lead to constraints
on the observed distribution -- some impose restrictions (solely) on
intervention distributions; see \citep{shpitser:14}.   A complete algorithm for
generating dormant constraints is given in \citep{shpitser08dormant}.

The remainder of the paper is organized as follows.  In Section
\ref{sec:cadmgs} we introduce Conditional Acyclic Directed Mixed
Graphs, the class of graphs we use to represent our models; those
models are formally introduced in Section \ref{sec:nested}.  Some
graphical theory is given in Section \ref{sec:part}, before the main
results in Section \ref{sec:main}.  Section \ref{sec:exm}
applies the method to data from a panel study.

\section{Conditional acyclic directed mixed graphs} \label{sec:cadmgs}

A directed acyclic graph (DAG) contains vertices representing random
variables, and edges (arrows) that imply some structure on the joint
probability distribution.  A DAG with latent vertices can be
transformed into an acyclic directed mixed graph (ADMG) over just its
observed vertices via an operation called \emph{latent projection}
\citep{pearl:92}.  In the simplest case this just
involves replacing latent variables with bidirected edges
($\leftrightarrow$), as illustrated by the transformation from Figure
\ref{fig:verma}(a) to (b); the transformed graph represents the
marginal distribution over the observed random variables $X_V$.

For technical reasons we work with a slightly larger class of graphs,
called \emph{conditional} acyclic directed mixed graphs (CADMGs).
These have two types of vertex, fixed ($W$) and random ($V$), and are
used to represent the structure of a set of distributions for $X_V$ 
indexed by possible values of $X_W$.

\begin{figure}
\begin{center}
\begin{tikzpicture}
[>=stealth, node distance=20mm]
\pgfsetarrows{latex-latex};
\begin{scope}
 \node[fv] (1) {1};
 \node[rv, right of=1] (2) {2};
 \node[rv, right of=2] (3) {3};
 \node[rv, right of=3] (4) {4};
 \node[rv, above of=3, yshift=-6mm] (5) {5};
 \draw[da] (1) -- (2);
 \draw[da] (2) -- (3);
 \draw[da] (3) -- (4);
 \draw[ba] (4) -- (5);
 \draw[ba] (2) -- (5);
\end{scope}
\end{tikzpicture}
\caption{A conditional acyclic directed mixed graph $\mathcal{L}$, with
  random vertices $V = \{2,3,4,5\}$ and fixed vertices $W = \{1\}$.}
\label{fig:graph2}
\end{center}
\end{figure}
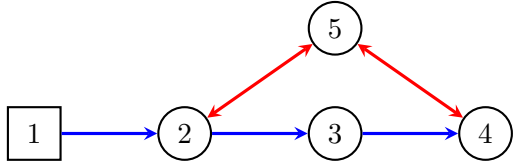

\begin{dfn}
  A \emph{conditional acyclic directed mixed graph} (CADMG) $\G$ is a
  quadruple $(V,W,\mathcal{E},\mathcal{B})$.  There are two disjoint
  sets of vertices: \emph{random}, $V$, and \emph{fixed}, $W$.  The
  \emph{directed edges} $\mathcal{E} \subseteq (V \cup W) \times V$
  are ordered pairs of vertices; if $(a,b) \in \mathcal{E}$ we write
  $a \rightarrow b$.  Loops $a \rightarrow a$ and directed cycles
  $a \rightarrow \cdots \rightarrow a$ are not allowed (hence
  `acyclic').

  The \emph{bidirected edges}, $\mathcal{B}$, are unordered pairs of
  distinct random vertices, and if $\{a,b\} \in \mathcal{B}$ we write
  $a \leftrightarrow b$. 
\end{dfn}

For convenience, throughout this paper we will only consider CADMGs in
which for every fixed vertex $w$ there is at least one edge $w\rightarrow v$.
(Note that it follows from the definition of a CADMG that $v$ will be random.)

These graphical concepts are most easily understood by example: see
the CADMG in Figure \ref{fig:graph2}.  
We depict random vertices with round nodes, and fixed vertices with
square nodes.  CADMGs are not generally simple
graphs, because it is possible to have up to two edges between each
pair of vertices in $V$ (one directed and one bidirected); see Figure 
\ref{fig:iv} for two examples.  CADMGs are
a slight generalization of ADMGs \citep{richardson:03}, which
correspond to the special case $W = \emptyset$.  Note that no 
arrowheads can be adjacent to any fixed vertex: so neither
$a \rightarrow w$ nor $a \leftrightarrow w$ is allowed for any
$w \in W$.  This reflects the fact that fixed vertices cannot depend
on other variables, observed or unobserved, but that random
vertices may depend upon fixed ones.  Mathematically, fixed nodes play
a similar role to the `parameter nodes' used by \citet{dawid:02}.

We make use of the following standard familial terminology for directed graphs.

\begin{dfn}  
  If $a \rightarrow b$ we say that $a$ is a \emph{parent} of $b$, and
  $b$ a \emph{child} of $a$.  The set of parents of $b$ is denoted
  $\pa_\G(b)$.  We say that $w$ is an \emph{ancestor} of $v$ if either
  $v=w$ or there is a sequence of directed edges
  $w \rightarrow \cdots \rightarrow v$.  The set of ancestors of $v$
  is denoted $\an_\G(v)$.  These definitions are applied disjunctively
  to sets of vertices so that, for example,
  $\pa_\G(A) \equiv \bigcup_{a \in A} \pa_\G(a)$.  An \emph{ancestral}
  set is one that contains all its own ancestors: $\an_\G(A) = A$.

Note that the definitions of parents, children and
ancestors do not distinguish between random and fixed vertices.
A \emph{random-ancestral} set, $A' \subseteq V$, is a set of random vertices such
that $\an_\G(A') \subseteq A' \cup W$; i.e.\ all the random ancestors of
$A'$ are contained in $A'$ itself.

A set of vertices $B$ is said to be \emph{sterile} if it does not
contain any of its children: $\ch_\G(B) \cap B = \emptyset$.  
The \emph{sterile subset} of a set $C \subseteq V$ is
$\sterile_\G(C) \equiv C \setminus \pa_\G(C)$ (sometimes 
called the set of `sink nodes' in the induced subgraph on ${C}$).
\end{dfn}

\begin{exm}
Consider the CADMG $\mathcal{L}$ in Figure
\ref{fig:graph2}.  The parents of the vertex 3 are
$\pa_{\mathcal{L}}(3) = \{2\}$, and its ancestors are
$\an_{\mathcal{L}}(3) = \{1,2,3\}$; hence $\{1,2,3\}$ is 
ancestral, and $\{2,3\}$ is random-ancestral.  The set $\{2,4,5\}$ is sterile,
but $\{2,3,5\}$ is not.
\end{exm}

\begin{dfn}
  A set of random vertices $B \subseteq V$ is
  \emph{bidirected-connected} if for each $a,b \in B$ there is a
  sequence of edges $a \leftrightarrow \cdots \leftrightarrow b$ with
  all intermediate vertices in $B$.  A maximal bidirected-connected
  set is a \emph{district} of the graph $\G$.  The set of districts is a partition of the 
  random vertices of a graph; the district containing $v \in V$ is
  denoted $\dis_\G(v)$.
\end{dfn}

We draw bidirected edges in red, which makes it easy to identify
districts as the maximal sets connected by red
edges.  In Figure \ref{fig:verma}(b) for example, there are three
districts: $\{X\}$, $\{M\}$, and $\{E,Y\}$.  In Figure \ref{fig:graph2}
there are two: $\{3\}$ and $\{2,4,5\}$.

\subsection{Transformations}

We now introduce two operations that transform CADMGs by removing
vertices: the first separates into districts and the second one forms ancestral subgraphs.
We will use these transformations to define our Markov property (and
thereby our statistical model) in Section \ref{sec:nested}.

\begin{dfn}
  Let $\G$ 
  be a CADMG containing a district $D$.  Define
  $\fd_D(\G)$ to be the CADMG with: the set of random vertices $D$; 
 the set of fixed vertices
  $\pa_\G(D) \setminus D$; 
  the set of bidirected edges whose endpoints are both in $D$ in $\G$;
  the set of  directed edges from $\G$ pointing
  to a vertex in $D$ (including directed edges between vertices in $D$).

  Let $A$ be a random-ancestral set in $\G$.  Define $\fm_A(\G)$ to be
  the graph with the set of random vertices $A$, the set of fixed vertices
  $\pa_\G(A) \setminus A$, and all edges between these vertices that are in $\G$.  Note
  that, since $A$ is random-ancestral, by definition the vertices in
  $\pa_\G(A) \setminus A$ are already fixed vertices in $\G$.

  If a graph $\mathcal{G}'$ can be obtained from $\G$ by iteratively
  applying operations of the form $\fd$ and $\fm$, we say that $\G'$
  is \emph{reachable} from $\G$.
\end{dfn}


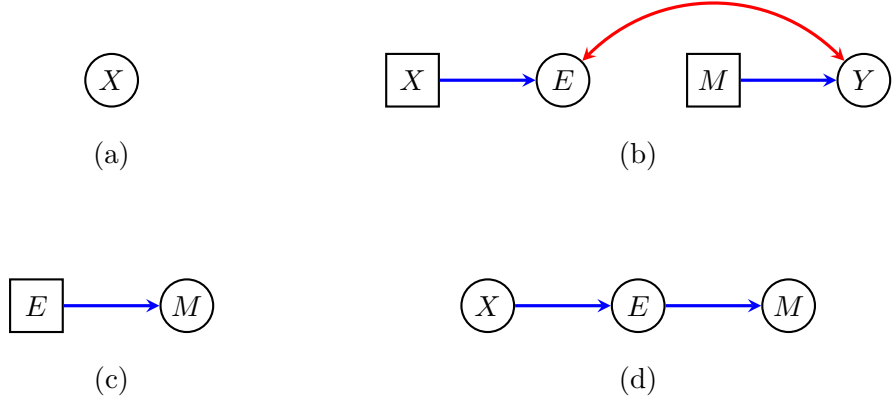
\begin{figure}
 \begin{center}
 \begin{tikzpicture}
 [node distance=20mm, >=stealth]
 \pgfsetarrows{latex-latex};
\begin{scope}[xshift=2cm]
 \node[rv]  (1)              {$X$};
 \node[below of=1, yshift=10mm] {(a)};
\end{scope}
\begin{scope}[yshift=-3cm, xshift=1cm]
 \node[rv, rectangle]  (2)              {$E$};
 \node[rv, right of=2] (3) {$M$};
 \draw[da] (2) -- (3);
 \node[below of=2, yshift=10mm, xshift=10mm] {(c)};
\end{scope}
\begin{scope}[xshift=7cm, yshift=-3cm]
 \node[rv]  (1)              {$X$};
 \node[rv, right of=1]  (2)              {$E$};
 \node[rv, right of=2] (3) {$M$};
 \draw[da] (1) -- (2);
 \draw[da] (2) -- (3);
 \node[below of=2, yshift=10mm] {(d)};
\end{scope}
\begin{scope}[xshift=6cm]
 \node[rv, rectangle]  (1)              {$X$};
 \node[rv, right of=1] (2) {$E$};
 \node[rv, right of=2, rectangle] (3) {$M$};
 \node[rv, right of=3] (4) {$Y$};
 \draw[da] (1) -- (2);
 \draw[da] (3) -- (4);
 \draw[ba] (2.45) .. controls +(1,1) and +(-1,1) .. (4.135);
  \node[below of=2, xshift=10mm, yshift=10mm] {(b)};
\end{scope}
 \end{tikzpicture}
 \caption{Reachable subgraphs of the graph in
   Figure \ref{fig:verma}(b).  Graphs in (a), (b) and (c) correspond to 
   factorization into the districts $\{X\}$,
   $\{E,Y\}$ and $\{M\}$ respectively.  Graph (d) corresponds to 
   marginalizing the childless node $Y$.
}
 \label{fig:subs}
 \end{center}
\end{figure}

\begin{figure}
\begin{center}
\begin{tikzpicture}
[>=stealth, node distance=20mm]
\pgfsetarrows{latex-latex};
\begin{scope}
 \node[fv] (1) {1};
 \node[rv, right of=1] (2) {2};
 \node[fv, right of=2] (3) {3};
 \node[rv, right of=3] (4) {4};
 \node[rv, above of=3, yshift=-6mm] (5) {5};
 \draw[da] (1) -- (2);
 \draw[da] (3) -- (4);
 \draw[ba] (4) -- (5);
 \draw[ba] (2) -- (5);
\node[below of=2, xshift=10mm, yshift=10mm] {(a)};
\end{scope}
\begin{scope}[yshift=-3cm]
 \node[fv] (1) {1};
 \node[rv, right of=1] (2) {2};
 \node[fv, right of=2] (3) {3};
 \node[rv, right of=3] (4) {4};
 \draw[da] (1) -- (2);
 \draw[da] (3) -- (4);
\node[below of=2, xshift=10mm, yshift=10mm] {(b)};
\end{scope}
\begin{scope}[yshift=-7cm]
  \node (1) {};
  \node[right of=1] (2) {};
 \node[fv, right of=2] (3) {3};
 \node[rv, right of=3] (4) {4};
 \node[rv, above of=3, yshift=-6mm] (5) {5};
 \draw[da] (3) -- (4);
 \draw[ba] (4) -- (5);
\node[below of=2, xshift=10mm, yshift=10mm] {(c)};
\end{scope}
\end{tikzpicture}
\caption{Three CADMGs reachable from the graph in Figure \ref{fig:graph2}.}
\label{fig:reached}
\end{center}
\end{figure}

\begin{exm}
  The graph in Figure \ref{fig:verma}(b) contains the districts
  $\{X\}$, $\{E,Y\}$ and $\{M\}$.  The corresponding graphs
  $\fd_{D}(\G)$ are given in Figure \ref{fig:subs}(a), (b) and
  (c) respectively.  The sets $\{X, E, M\}$,  $\{X, E\}$ and $\{X\}$ are ancestral in $\G$, and the
  graphs $\fm_{\{X\}}(\G)$ and $\fm_{\{X, E, M\}}(\G)$ are shown in Figures
  \ref{fig:subs}(a) and (d) respectively.
\end{exm}

\begin{exm}
  The graph in Figure \ref{fig:graph2} contains the district
  $\{2,4,5\}$, and $\fd_{\{2,4,5\}}(\mathcal{L})$ gives us the graph
  in Figure \ref{fig:reached}(a).  The sets $\{2,4\}$ and $\{4,5\}$
  are both random-ancestral in $\fd_{\{2,4,5\}}(\mathcal{L})$, so we can
  apply either $\fm_{\{2,4\}}$ or $\fm_{\{4,5\}}$ to obtain the CADMGs
  in Figures \ref{fig:reached}(b) and (c) respectively.
\end{exm}

As we will see in the next section, both of these graphical operations
correspond to an operation on a probability distribution we associate
with the graph: $\fm_A$ to marginalization, and $\fd_D$ to a
factorization.  The `fixing' operation described in
\citet{shpitser:17} unifies and generalizes $\fm$ and $\fd$, but the
statistical model we will describe is ultimately the same.  For the
purposes of defining a parameterization it is more convenient to use
the formulation given here.

Note that if we start with a graph $\G$ in which all the fixed
vertices $w \in W$ have at least one child, then this is also true of
the graph obtained after applying either $\fm_A$ or $\fd_D$.  

It is important to note that sets may become districts or random
ancestral sets after several iterations of $\fm$ and $\fd$.  For
example, $\{2,4\}$ is not random-ancestral in $\mathcal{L}$, but it is
in $\fd_{\{2,4,5\}}(\mathcal{L})$.  Similarly $\{4\}$ is not a
district in $\fd_{\{2,4,5\}}(\mathcal{L})$, but it is in
$\fm_{\{2,4\}}(\fd_{\{2,4,5\}}(\mathcal{L}))$; see Figure
\ref{fig:reached}(b).

We now give a characterization of what reachable
graphs look like.

\begin{dfn} \label{dfn:reach}
Let $\G$ be a CADMG with random vertex set $V$.  Given 
$C \subseteq V$ the graph $\G[C]$ is defined to be the CADMG with 
the set of random vertices $C$, fixed vertices $\pa_\G(C) \setminus C$,
those bidirected edges in $\G$ with both endpoints in
  $C$, and those directed edges that are directed from $C \cup
  \pa_\G(C)$ to $C$.
\end{dfn}

In other words, $\G[C]$ is the subgraph containing
precisely the edges whose arrowheads are all in $C$.   For example,
if $\G$ is the graph in Figure \ref{fig:verma}(b), then Figure \ref{fig:subs}(a)--(d)
corresponds to $\G[\{X\}]$, $\G[\{E,Y\}]$, $\G[\{M\}]$ and $\G[\{X,E,M\}]$ 
respectively. 

\begin{lem} \label{lem:welldef} Suppose that the graph $\G'$ is
  reachable from $\G$ and has set of random vertices $C$.  Then $\G' = \G[C]$.
\end{lem}

\begin{proof}
  Since we assume all fixed vertices have at least one child, then
  $\G = \G[V]$.  In addition it is clear from the definitions of $\fd$
  and $\fm$ that precisely the edges and fixed vertices mentioned are
  preserved at each step.
\end{proof}

In the rest of this paper we will only refer to $\G[C]$ if $C$ is a 
reachable set, though Definition \ref{dfn:reach} in principle applies
to any $C \subseteq V$. 
Unfortunately there is generally no simple way of characterizing which 
sets $C$ correspond to reachable subgraphs without iteratively 
applying $\fd$ and $\fm$ as defined above.  
If a set $A$ is random-ancestral, then clearly $\G[A]$ is reachable
just by applying $\fm$.
Note that \citet{shpitser:17} use a slightly
more general definition of reachable sets.  


\section{Nested Markov Property} \label{sec:nested}

Graphical models relate the structure of a graph to a collection of
joint probability distributions over a set of random variables.  We will 
work with the nested Markov property, which 
relates a (C)ADMG and each of its reachable subgraphs to
a collection of probability distributions over random vertices, indexed by fixed vertices.

Suppose we are interested in random variables $X_v$ taking values in a
finite discrete set $\X_v$.  For a set of vertices $C$ let
$\X_{C} \equiv \times_{v \in C}\, \X_v$.  A \emph{probability 
kernel for $V$ given $W$} (or simply a kernel) is a
function $p_{V|W} : \X_V \times \X_W \rightarrow [0,1]$ such that for
each $x_W \in \X_W$,
\begin{align*}
\sum_{x_V \in \X_V} p_{V|W}(x_V \,|\, x_W) = 1.
\end{align*}
In other words, a kernel behaves like a conditional probability
distribution for $X_V$ given $X_W$.  We use the word `kernel' to
emphasize that some of the conditional distributions we obtain are not
equal to the usual conditional distribution obtained from elementary
definitions, but instead correspond to certain interventional
quantities.  

In what follows, $\dot\cup$ is used to denote a union of disjoint
sets.

\begin{dfn}
Let $p_{V|W}$ be a kernel, and let $A \dot\cup B \dot\cup C = V$.  The \emph{marginal
kernel} over $A, B \,|\, W$ is defined to be:
\begin{align*}
p_{AB|W}(x_A, x_B \,|\, x_W) \equiv \sum_{x_{C}} p_{V|W}(x_V \,|\, x_W).
\end{align*}
It is easy to check that $p_{AB|W}$ is also a kernel.
A (version of the) \emph{conditional kernel} of $A|B,W$ is any kernel $p_{A|BW}$
satisfying
\begin{align*}
p_{A|BW}(x_A \,|\, x_B, x_W) \cdot p_{B|W}(x_B \,|\, x_W)  \equiv\, p_{AB|W}(x_A, x_B \,|\, x_W).
\end{align*}
\end{dfn}

This is uniquely defined precisely for $x_B, x_W$ such that
$p_{B|W}(x_B \,|\, x_W) > 0$.

\begin{rmk} \label{rmk:equiv}
  Note that, for convenience, if some of the fixed variables
  $W^* \subseteq W$ in a kernel $p_{V|W}$ are entirely irrelevant,
  (i.e.\ if the functions $p_{V|W}(\cdot \,|\, \cdot, y_{W^*})$ are
  identical for all $y_{W^*} \in \X_{W^*}$) we will describe it interchangably
  as a kernel of $V$ given $W$, and as a kernel of $V$ given
  $W \setminus W^*$, since in this case these objects are isomorphic:
$p_{V|W} =   p_{V|W \setminus W^*}$.
  %
\end{rmk}

We are now in a position to define the nested model.  The definition is recursive, and
works by reference to the model applied iteratively to smaller and smaller graphs.
The model is introduced in \citet{shpitser:17}, and is based on the constraint finding
algorithm of \citet{tian:02a}, which follows a similar recursive structure.

\begin{dfn} \label{dfn:rf}
  Let $\G$ be a CADMG and $p_{V|W}$ a
  probability kernel.  Say that $p_{V|W}$
  \emph{recursively factorizes} according to $\G$, and write $p_{V|W}
  \in \RF(\G)$ if either $|V| = 1$, or both:
\benum[(i)]
\item if $\G$ has districts $D_1, \ldots, D_k$, $k \geq 1$ then 
\begin{equation}
p_{V|W}(x_V\,|\,x_W) =
  \prod_i r_i(x_{D_i}\,|\, x_{\pa(D_i)\setminus D_i}) \label{eqn:rec-fact}
\end{equation}
and where, if $k \geq 2$, each $r_i$ recursively factorizes according to $\G[D_i] = \fd_{D_i}(\G)$; and
\item for each ancestral set $A$ with $V \setminus A \neq \emptyset$, 
the marginal distribution
\[
p_{A \cap V|W}(x_{A \cap V} \,|\, x_{W}) = \sum_{x_{V\setminus A}} p_{V|W}(x_V \,|\, x_W)
\]
does not depend upon $x_{W \setminus A}$ (so we denote it by $p_{A \cap V|A \cap W}$
in line with Remark \ref{rmk:equiv}), and this
recursively factorizes according to $\G[V \cap A] = \fm_{V\cap A}(\G)$.
\eenum
\end{dfn}

Given a graph $\G$, we shall refer to $\RF(\G)$ as the \emph{nested model}
associated with $\G$, and say that distributions in that set satisfy the \emph{nested
Markov property} with respect to $\G$.  There are other, equivalent definitions: 
see \citet{shpitser:17}.

\begin{rmk}
  It is important to note that, in terms of the factors $r_i$ whose existence is implied
  by condition (i), the definition of recursive
  factorization `starts from scratch' each time we perform the
  recursion.  For example, we make no claim (yet) about the connection
  between a factor $r_i$ obtained from (i) and any such factors which
  arise after first applying (ii) and then later (i): see Example \ref{exm:order}.

  In the base case $V = \{v\}$ the definition places no restriction on the 
  distribution of $X_v$ given $X_W$.
	The observed distribution
	obtained from a directed acyclic graph model with latent variables
	will satisfy conditions (i) and (ii) with respect to the ADMG
	that is the latent projection of that DAG
  \citep{tian:02, tian:02a}.  The models defined by the Markov
  properties for ADMGs introduced by \citet{richardson:03} and
  parameterized by \citet{evans:13,evans:14} can be defined by
  replacing (i) with the weaker requirement:
  \begin{enumerate}[(i')]
\item $\G$ has districts $D_1, \ldots, D_k, k \geq 1$, and $p_{V|W} =
  \prod_i r_i$ where each $r_i$ is a kernel for $D_i$ given 
  $\pa_\G(D_i) \setminus D_i$.
  \end{enumerate}
  In other words, although the distribution must satisfy the
  ancestrality condition (ii) and then factorize, no further conditions are
  imposed on those factors: they are not required to obey any
  additional constraints implied by the graph $\G[D_i]$.  This leads to a model defined
  entirely by conditional independence relations on the original joint
  distribution $p_{V|W}$.
  
  As a consequence of this, the m-separation criterion for ordinary Markov models 
  \citep[as well as the other Markov properties described by][]{richardson:03} 
  can be applied to correctly to the initial ADMG $\G$ to derive conditional 
  independences in $p(x_V)$; however,
  it does not completely describe the nested model.  
\end{rmk}

\begin{exm} \label{exm:order}
  Consider the CADMG in Figure \ref{fig:verma}(b).  Criterion (i) of
  recursive factorization requires that
\begin{align*}
p(x, e, m, y) = r_{X}(x) \cdot r_{EY}(e, y \,|\, x, m) \cdot r_{M}(m \,|\, e) 
\end{align*}
for distributions $r_{X}$,  $r_{EY}$ and $r_{M}$ which recursively
factorize according to the graphs in Figure \ref{fig:subs}(a), (b) and
(c) respectively.  

On the other hand, if we apply condition (ii) to the childless node $Y$
we see that the margin $p(x, e, m)$ must satisfy recursive factorization 
with respect to the DAG in Figure \ref{fig:subs}(d), so
\begin{align*}
p(x, m, e) = \tilde{r}_{X}(x) \cdot \tilde{r}_{E}(e \,|\, x) \cdot \tilde{r}_{M}(m \,|\, e)
\end{align*}
for some kernels $\tilde{r}_{X}$, $\tilde{r}_{M}$ and
$\tilde{r}_E$.  This factorization implies the conditional independence
$X \indep M \,|\, E$, which can also be deduced using m-separation.  
We add a tilde to the kernels to emphasise that the
\emph{definition} starts afresh at each iteration, and makes no claim of any
relationship between this factorization and the factorization of
$p = r_X r_{EY} r_M$.  However, it is not hard to verify that in fact
\begin{align*}
r_{X}(x) &= \tilde{r}_{X}(x) = p(x),\\
r_{M}(m \,|\, e) &= \tilde{r}_{M}(m \,|\, e) = p(m \,|\, e),\\
\sum_y r_{EY}(e,y \,|\, x, m) &= \tilde{r}_E(e \,|\, x) = p(e \,|\, x).
\end{align*} 
In fact it will follow from Theorem \ref{thm:main} that, in general,
kernels such as $r_X$ and $\tilde{r}_X$ that have the same random vertex 
set but are derived in different ways are equal under the model. 
Note that
\begin{align*}
r_{EY}(e, y \,|\, x, m) &= p(e \,|\, x) \cdot p(y \,|\, x, m, e)\\
&\neq p(e \,|\, x, m) \cdot p(y \,|\, x, m, e)\\
&= p(e, y \,|\, x, m),
\end{align*}
and so $r_{EY}$ is \emph{not} the usual conditional distribution
of $E, Y$ given $X, M$.
\end{exm}

\subsection{Properties of the Recursive Kernels}

Here we show that the kernels $r_i$ from (\ref{eqn:rec-fact}) 
in Definition \ref{dfn:rf} are products of conditional 
distributions derived from $p_{V|W}$ at the
current level of the recursion, and that they are uniquely defined up to
versions of those conditional distributions.

A \emph{topological ordering} of the random vertices of a CADMG is a
total ordering $<$ on $V$ such that every vertex precedes its
children.  We denote by $\pre_{<}(v)$ the set of (random) vertices
which precede $v$ under $<$.

The following proposition shows that the factors in the definition of
recursive factorization are unique up to versions of conditional
distributions.


\begin{prop} \label{prop:gform} Let $\G$ be a CADMG with districts
  $D_1, \ldots, D_k$, and let $<$ be any topological ordering of $V$.
  Let $p_{V|W} = \prod_i r_i$, where each $r_i$ recursively factorizes
with respect to $\G[D_i]$.  Then 
\begin{align}
r_i(x_{D_i} \,|\, x_{\pa_\G(D_i) \setminus D_i}) = \prod_{v \in D_i} p_{v|\pre_<(v) \cup W}(x_v \,|\, x_{\pre_<(v)} , x_W),
\label{eqn:gform}
\end{align}
where $p_{v|\pre_<(v) \cup W}$ is any $p_{V|W}$-version of the conditional
distribution of $X_v | X_{\pre_<(v)}, X_W$.
%
\end{prop}

\begin{rmk}
The equation in (\ref{eqn:gform}) is an instance of the \emph{g-formula} of \citet{robins:86}. 
The result also appears as Corollary 1 in \citet[][Section 4.3]{tian:thesis}, in the case of latent 
variable models.
\end{rmk}

\begin{proof}
  For the purposes of induction we generalize the result slightly to
  allow $D_i$ to be collections of several districts.  Let
  $E_i \equiv \pa_\G(D_i) \setminus D_i$.  We proceed by induction on
  $|V|$: if $|V| \leq 1$ there is nothing to show.  Otherwise, let
  $t \in D_k$ be the last vertex in the ordering $<$, so that $x_t$
  only appears as a variable in the factor $r_k$.  Then
\begin{align*}
p_{V \setminus \{t\}|W}(x_{V \setminus \{t\}} \,|\, x_W) &\equiv \sum_{x_t} p_{V|W}(x_V \,|\, x_W)\\ 
&= \sum_{x_t} \prod_{i=1}^k r_i(x_{D_i} \,|\, x_{E_i})\\
&= \left( \sum_{x_t} r_k(x_{D_k} \,|\, x_{E_k}) \right) \prod_{i=1}^{k-1} r_i(x_{D_i} \,|\, x_{E_i}) \\
&= \tilde{r}_k(x_{D_k \setminus \{t\}} \,|\, x_{E_k}) \prod_{i=1}^{k-1} r_i(x_{D_i} \,|\, x_{E_i})
\end{align*}
where, by property 1 of recursive factorization, the kernel
$\tilde{r}_k$ recursively factorizes with respect to the graph
$\G[D_k \setminus \{t\}]$.  Similarly all the factors $r_i$ for
$i=1,\ldots,k-1$ recursively factorize with respect to $\G[D_i]$, so
by the induction hypothesis each such $r_i$ is of the required form
(\ref{eqn:gform}), and
\begin{align*}
\tilde{r}_k(x_{D_k \setminus \{t\}} \,|\, x_{E_k}) &= \prod_{v \in D_k \setminus \{t\}} p_{v|\pre_<(v) \cup W}(x_v \,|\, x_{\pre_<(v)} , x_W).
\end{align*}
But then
%
\begin{align}
\prod_i r_i = p_{V|W} &= p_{t|VW \setminus \{t\}} \cdot p_{V \setminus \{t\}|W} = p_{t|VW \setminus \{t\}} \cdot \tilde{r}_k \cdot \prod_{i=1}^{k-1} r_i; \label{eqn:reprod}
\end{align}
therefore 
\begin{align*}
r_k(x_{D_k} \,|\, x_{E_k}, x_W) &= p_{t|VW \setminus \{t\}}(x_t \,|\, x_{V \setminus \{t\}}, x_W) \cdot \tilde{r}_k
\end{align*}
and $p_{t|VW \setminus \{t\}}$ satifies (\ref{eqn:reprod}) if and only if
it is a version of the relevant conditional distribution, as required.
\end{proof}

The next result shows that the positivity of $p_{V|W}$ is preserved in
any derived kernels.

\begin{lem} \label{lem:pos} 
  Let $p_{V|W}(x_V \,|\, x_W)$ be a probability distribution, $<$ some total
ordering on $V$, and let $A \subseteq V$ and $B \equiv W \cup \pre_<(A) \setminus A$.
Define
\begin{align*}
r_{A|B}(x_{A} | x_{B}) \equiv \prod_{v \in A} p_{v|\pre_<(v),W}(x_{v} | x_{\pre_<(v)}, x_W),
\end{align*}
for some versions $p_{v|\pre_<(v),W}$ of the conditional distributions of $X_v \,|\, X_W, X_{\pre_<(v)}$.

Then:
\begin{enumerate}[(a)]
\item $r_{A|B}$ is a kernel for $X_A \,|\, X_B$;
\item for any $T \subseteq V$, $x_T \in \X_T$ and $x_W \in \X_W$, if
  $p_{T|W}(x_{T} \,|\, x_W) > 0$ then
  \[
  r_{T\cap A | B}(x_{T\cap A} \,|\, x_B) \equiv \sum_{y_{A \setminus T}} r_{A | B}(y_{A \setminus T}, x_{T \cap A} \,|\, x_B) > 0
  \] and all versions of
  $r_{T\cap A | B}(x_{T\cap A} \,|\, x_B)$ are the same;
\item if $p_{T|W}(x_{T} \,|\, x_W) = 0$ then there exists $t \in T$
  such that (every version of)
  \[
p_{t|\pre_<(t), W}(x_t \,|\, x_{\pre_<(t)}, x_W) = 0.
\]
\end{enumerate}
\end{lem}

\begin{proof}
(a) Clearly $r_{A|B} \geq 0$ since it is a product of conditional distributions, which are 
themselves non-negative.  In addition, by summing the expression above in reverse order of $<$
it is easy to see that $\sum_{x_A} r_{A|B}(x_A \,|\, x_B) = 1$ for any $x_B \in \X_B$.  Hence $r_{A|B}$
is a kernel.

For (b) note that if $p_{T|W}(x_{T} \,|\, x_W) > 0$ then there exists
some $x_{V \setminus T} \in \X_{V \setminus T}$ such that $p_{V|W}(x_{V} \,|\, x_W) > 0$.  Then
\begin{align*}
p_{V|W}(x_V \,|\, x_W) &= \prod_{v \in V} p_{v|\pre_<(v),W}(x_v \,|\, x_{\pre_<(v)}, x_W)\\
  &= r_{A|B}(x_A \,|\, x_B) \prod_{v \in V \setminus A} p_{v|\pre_<(v),W}(x_v \,|\, x_{\pre_<(v)}, x_W),
\end{align*}
so if the left hand side is positive then so is
$r_{A|B}(x_A \,|\, x_B) > 0$.  Since all the events in this expression
have positive $p_{V|W}$ probability, all versions of each conditional
probability are equal.

Lastly, if $p_{T|W}(x_{T} \,|\, x_W) = 0$ then clearly
some factor of 
\begin{align*}
0 = p_{T|W}(x_{T} \,|\, x_W) = \prod_{t \in T} p_{t|\pre_<(t),W}(x_{t} \,|\, x_{\pre_<(t)}, x_W)
\end{align*}
is also zero.  Pick the $<$-minimal $t$ such that this holds, so that 
$p_{\pre_<(t) \,|\, W}(x_{\pre_<(t)} \,|\, x_W) > 0$.  Then (c) holds.
\end{proof}

A corollary of this lemma is that, if $p_{V|W}$ is strictly positive,
the kernels  $r_i$ derived from it by application of Definition
\ref{dfn:rf} are uniquely defined.


\begin{cor}
  Let $p_{V|W} \in \RF(\G)$ be a strictly positive kernel.  
  Then any kernel derived from $p_{V|W}$ by repeated applications 
  of Definition \ref{dfn:rf} (using $\G$) is uniquely defined.
\end{cor}

\begin{proof}
  Clearly applying (ii) is always unique, since it only involves
  summing.  By Proposition \ref{prop:gform}, application of (i) is a
  factorization into univariate conditional distributions, each of
  which is uniquely defined when the joint distribution is positive.
  In addition, by Lemma \ref{lem:pos} each such conditional
  distribution is also strictly positive, so following the recursion
  with each unique factor gives the result.
\end{proof}

\section{Intrinsic Sets and Partitions} \label{sec:part}

In this section we provide the necessary theory to link the graphical
notions of Section \ref{sec:nested} to the parameterization in Section
\ref{sec:main}.  The parameterization uses factorizations of the
distribution into pieces which correspond to special subsets of
vertices in the graph; these subsets are themselves derived from the
idea of the `reachable' sets already introduced.

\begin{dfn} \label{dfn:int}
  Let $\G$ be a CADMG.  A non-empty set $S$ of random vertices is
  \emph{intrinsic} if it is bidirected-connected and the graph $\G[S]$
  is reachable from $\G$.  

  For each intrinsic set $S$, define the
  associated \emph{recursive head} by $\rh_\G(S) = \sterile_\G(S)$;
  i.e.\ it is the set of sink nodes in the induced subgraph over 
  $S$.  The set of recursive heads is denoted by $\mathcal{H}(\G)$, 
  or simply $\mathcal{H}$.\footnote{Note 
  that the definition of a recursive head differs from
  the \emph{head} used in \citet{evans:14} for ADMGs.  We 
  will see in Example \ref{exm:vermahead} that $\{E,Y\}$ is a recursive head
  in the graph in Figure \ref{fig:verma}(b), but one can check that it 
  is not a head in the \citet{evans:14} sense.}  
  
  The \emph{tail} associated with a recursive head $H$ 
  (and the relevant intrinsic set $S$) is $T(H) \equiv \pa_\G(S)$.  
  We will denote a tail by $T$ if it is unambiguous which recursive head it is 
  derived from.
\end{dfn}

Intrinsic sets are central to the nested Markov property as they are
the sets of variables over which the kernels $r_i$ in Definition 
\ref{dfn:rf} specify distributions.  
Intrinsic sets do not appear to be easily characterized in terms of the 
presence of a  path in the original graph; Definition \ref{dfn:int}
 implicitly considers a sequence of graphs
generated via repeated applications of the two operations $\fd$
  and $\fm$. The set of intrinsic sets may be found
in polynomial time; see \citep{shpitser:11}.

\begin{exm}
For the graph $\mathcal{L}$ in Figure \ref{fig:graph2}, $\{2,4,5\}$
and $\{3\}$ are districts and therefore intrinsic sets.  The graph
$\mathcal{L}[\{2,4,5\}]$ is shown in Figure \ref{fig:reached}(a);
applying $\fm$ appropriately to random-ancestral sets yields all the other
intrinsic sets: $\{2,5\}, \{4,5\}, \{2\}, \{4\}$ and $\{5\}$.  
Each recursive head is equal to the associated intrinsic set.
\end{exm}

\begin{dfn}
  Let $B \subseteq V$ be a set of random vertices in $\G$.  Suppose we alternately 
  marginalize
  vertices that are not ancestors of $B$, and remove those which are
  not in the same district as some element of $B$:
\begin{align}
& \G \mapsto \fm_{\an_\G(B)}(\G), && \G \mapsto \fd_{\dis_\G(B)}(\G). \label{eqn:iter}
\end{align}

If these two operations change anything at all then they reduce the
size of the set of random vertices; consequently repeatedly applying both
these operations successively will
eventually reach some stable point, which is a graph whose set of 
random vertices we denote by $I_\G(B)$.  Note that at each step of
(\ref{eqn:iter}) the random vertices in the resulting graph always
include $B$, so $B \subseteq I_\G(B)$.  

If $I_\G(B)$ is bidirected-connected then it is an intrinsic set
by definition, and we call $I_\G(B)$ the \emph{intrinsic closure} 
of $B$.  
\end{dfn}

\begin{prop} \label{prop:preserve}
If $\G' = \G[C]$ is reachable from $\G$ for some set $C \supseteq B$, then
\begin{align*}
\fm_{\an_{\G'}(B)}(\G') \subseteq \fm_{\an_{\G}(B)}(\G) && \fd_{\dis_{\G'}(B)}(\G') \subseteq \fd_{\dis_{\G}(B)}(\G).
\end{align*}
\end{prop}

\begin{proof}
From Lemma \ref{lem:welldef}, $\G' = \G[C]$; 
We have $\fm_{\an_{\G'}(B)}(\G') = \G[\an_{\G'}(B)]$ and $\fm_{\an_{\G}(B)}(\G) = \G[\an_{\G}(B)]$.  
Any ancestor of $B$ in the subgraph $\G'=\G[C]$ must be  also be an ancestor in $\G$,
so clearly $\G[\an_{\G'}(B)] \subseteq \G[\an_{\G}(B)]$.  A similar argument holds
for $\fd$.  
\end{proof}

%

Both of the operators in (\ref{eqn:iter}) are idempotent; in addition, since
the sets $\an(B)$ and $\dis(B)$ only get smaller through repeated 
iterations, it follows from Proposition \ref{prop:preserve} that the stable point does not 
depend upon which operation is applied first.  Hence $I_\G(B)$ is 
well-defined. 

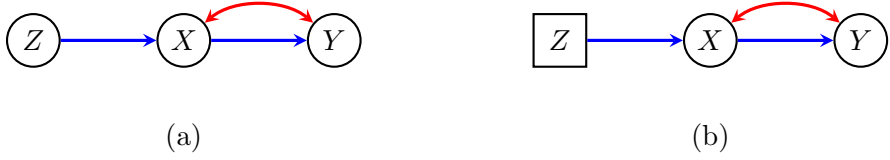
\begin{figure}
\begin{center}
\begin{tikzpicture}
[>=stealth, node distance=20mm]
\pgfsetarrows{latex-latex};
\begin{scope}
 \node[rv] (1) {$Z$};
 \node[rv, right of=1] (2) {$X$};
 \node[rv, right of=2] (3) {$Y$};
 \draw[da] (1) -- (2);
 \draw[da] (2) -- (3);
\draw[ba] (2.40) .. controls +(40:.5) and +(140:.5) .. (3.140);
\node[below of=2, yshift=7mm] {(a)};
\end{scope}
\begin{scope}[xshift=7cm]
 \node[rv, rectangle] (1) {$Z$};
 \node[rv, right of=1] (2) {$X$};
 \node[rv, right of=2] (3) {$Y$};
 \draw[da] (1) -- (2);
 \draw[da] (2) -- (3);
\draw[ba] (2.40) .. controls +(40:.5) and +(140:.5) .. (3.140);
\node[below of=2, yshift=7mm] {(b)};
\end{scope}
\end{tikzpicture}
\caption{(a) A (C)ADMG $\G$ and (b) $\G_1 \equiv \fd_{\dis(Y)}(\G)$.}
\label{fig:iv}
\end{center}
\end{figure}

\begin{exm}
Let $\G$ be the graph in Figure \ref{fig:iv}(a) and consider the 
intrinsic closure of the bidirected-connected set $\{Y\}$.  The graph $\fm_{\an(Y)}(\G)$
is just $\G$, since everything is an ancestor of $Y$.  However
$\G_1 \equiv \fd_{\dis(Y)}(\G)$ gives the graph $\G[\{X,Y\}]$
shown in Figure \ref{fig:iv}(b) 
in which $Z$ is fixed, but the edges are all unchanged.
It then becomes clear that repeatedly applying $\fm$ and $\fd$ 
will not result in any further changes to the graph.
Hence the intrinsic closure is just the set of random vertices in this
graph: $I_\G(\{Y\}) = \{X,Y\}$.

On the other hand, consider the graph $\mathcal{L}$ in Figure
\ref{fig:graph2} and the intrinsic closure of the set $\{4,5\}$.  
Again $\fm_{\an(\{4,5\})}(\mathcal{L}) = \mathcal{L}$, and
then $\fd_{\dis(\{4,5\})}(\mathcal{L})$ gives the graph in 
Figure \ref{fig:reached}(a).  Applying $\fm_{\an(\{4,5\})}(\cdot)$ 
to this graph yields the graph in Figure \ref{fig:reached}(c), 
whose only random vertices are $\{4,5\}$.  Hence, the procedure
terminates and, since it forms a district in this graph, 
$\{4,5\}$ is an intrinsic set and its own intrinsic closure. 
\end{exm}

One consequence of the next result is that, as we would 
hope, every intrinsic set is its own intrinsic closure.

\begin{lem} \label{lem:intclos} Let $S$ be an intrinsic set with
  recursive head $H$ in a graph $\G$.  Then for any $H \subseteq A
  \subseteq S$ we have $I_\G(A) = S$.
\end{lem}

\begin{proof}
  By the definition of $H$, every vertex in $S$ is either in $H$ or is
  a parent of some other element of $S$.  Since $S$ is
  bidirected-connected, the operations $\fd_A$, $\fm_A$ therefore
  cannot remove any element of $S$ without also having removed an
  element of $H$, but this is not allowed since $H \subseteq A$.
  Hence no element of $S$ is ever removed, and $I_\G(A) \supseteq S$.

  Suppose that $I_\G(A) \supset S$ and so
  $B \equiv I_\G(A) \setminus S$ is non-empty.  Every element of $B$
  is an ancestor of some other entry in $I_\G(A)$.  In addition, every
  element of $I_\G(A)$ is connected to $A \subseteq S$ by sequences of 
  bidirected edges through $I_\G(A)$, so $I_\G(A)$ is, like $S$, a
  bidirected-connected set.  
 Thus we cannot remove any element of $B$ via
  operations of the form $\fm, \fd$ without first removing some
  element of $A \subseteq S$.  If $B$ is non-empty then this
  implies $S$ is not reachable, which
  contradicts the assumption that $S$ is intrinsic.
\end{proof}

Note that a corollary of this result is that recursive heads are
in one-to-one correspondence with intrinsic sets: two distinct 
intrinsic sets may not have the same recursive head.

\begin{prop} \label{prop:conn}
If $B$ is a bidirected-connected set with intrinsic closure $I_\G(B)$, 
then the recursive head $H$ associated with the intrinsic set $I_\G(B)$ 
satisfies $H \subseteq B$.
\end{prop}

\begin{proof}
By definition of intrinsic closure, every vertex $v$ in $I_\G(B)$ is an ancestor of
$B$ in $\G[I_\G(B)]$. If $v \notin B$ then $v \notin  \sterile_\G(I_\G(B))$, hence
$v \notin H$.
%
\end{proof}

\begin{lem} \label{lem:singlehead}
Every singleton $\{v\}$ for $v \in V$ is a recursive head.
\end{lem}

\begin{proof}
Take the intrinsic closure $I_\G(\{v\})$ of $v$.
Every element of $I_\G(\{v\})$ other than $v$ is a parent of some 
other element of $I_\G(\{v\})$ by definition; therefore $\{v\}$ is 
the sterile set, and a recursive head.
\end{proof}

\begin{lem} \label{lem:intpres}
  Let $\G$ be a CADMG, and $\G'$ be a CADMG with random vertices $V'$,
  reachable from $\G$.  Then the intrinsic sets of $\G'$ 
  are precisely the intrinsic sets of $\G$ that are contained in $V'$,
  and their associated recursive heads and tails are the same.
\end{lem}

\begin{proof}
  Since $\G' = \G[V']$ is reachable from $\G$, any intrinsic set in
  $\G'$ is also an intrinsic set in $\G$.  For the converse, suppose
  that $D \subseteq V'$ is an intrinsic set in $\G$.  Take the intrinsic
  closure of $D$ in $\G'$, say $C$; if $C=D$ then we are done.

  Suppose not, so that $C \setminus D$ is non-empty.  This occurs
  precisely when $C$ is bidirected-connected in $\G^\prime$, and every vertex in
  $C \setminus D$ is an ancestor in $\G^\prime$ of some other vertex in $C$.  But if
  this is true in $\G'$ then it must also be true in $\G$, which
  contains any edges that $\G'$ does; thus the intrinsic closure of
  $D$ in $\G$ is a strict superset of $D$.  This contradicts the assumption that 
  $D$ is intrinsic in $\G$.

  By Lemma \ref{lem:welldef} the recursive heads and tails associated with each
  intrinsic set are unchanged, since the parent sets of each random
  vertex are preserved.
\end{proof}

\begin{cor} \label{cor:anccap} Let $\G$ be a CADMG containing
  random-ancestral sets $A_1, A_2$.  If $H \in \mathcal{H}(\G[A_1])$
  and $H \in \mathcal{H}(\G[A_2])$ then
  $H \in \mathcal{H}(\G[A_1 \cap A_2])$.
\end{cor}

\begin{proof}
  If $A_1$ and $A_2$ are random-ancestral, then so is $A_1 \cap A_2$,
  so the graph $\G[A_1 \cap A_2]$ is reachable from $\G$.  The result
  follows from Lemma \ref{lem:intpres}.
\end{proof}

\subsection{Partitions}

We follow the approach of \citet{evans:14} by defining partitions of
sets via appropriate collections of subsets.  Define a partial
ordering $\prec$ on recursive heads by $H_1 \prec H_2$ whenever
$I_\G(H_1) \subset I_\G(H_2)$.  

\begin{dfn}
Define a function $\Phi_\G$ on sets of random vertices
$C \subseteq V$ that `picks out' the set of $\prec$-maximal recursive
heads $H \in \mathcal{H}(\G)$ that are subsets of $C$.  That is,
\begin{align*}
\Phi_\G(C) \equiv \{H \in \mathcal{H} \,|\, H \subseteq C \text{ and } H \nprec H' \text{ for all other } H' \subseteq C, H' \in \mathcal{H}\}.
\end{align*}
Define
\begin{align*}
\psi_\G(C) \equiv C \setminus \bigcup_{D \in \Phi_\G(C)} D.
\end{align*}
Now recursively define a function $\partn{\cdot}_\G$ that partitions subsets of 
$V$: define $\partn{\emptyset}_\G = \emptyset$, and 
%
\begin{align*}
\partn{W}_\G \equiv \Phi_\G(W) \cup \partn{\psi_\G(W)}_\G.
\end{align*}
\end{dfn}

For full details, including a proof that this definition does
indeed define a partition, see the Appendix B.  

\begin{exm} \label{exm:vermahead}
  The recursive heads of the graph in Figure \ref{fig:verma}(b) are
  $\{X\}$, $\{E\}$, $\{M\}$, $\{Y\}$, $\{E,Y\}$, and the ordering requires
  that $\{E\}$ and $\{Y\}$ precede $\{E,Y\}$.  Hence, for example
  \begin{align*}
\partn{\{X,E,Y\}}_\G &= \{\{X\}, \{E, Y\}\},\\
\partn{\{M,Y\}}_\G &= \{\{M\}, \{Y\}\}.
  \end{align*}
The partitioning function $[ \cdot]_\G$ 
in \citet{evans:14} made use of `heads' rather than `recursive heads',
and therefore the partition obtained differs from the one here.  For example,
applied to the same graph as above, 
  \begin{align*}
[\{X,E,Y\}]_\G &= \{\{X\}, \{E\}, \{Y\}\}.
  \end{align*}
\end{exm}

\begin{lem} \label{lem:headpreserve}
  If $\G' = \G[D]$ is reachable from $\G$ then
  $\partn{C}_{\G'} = \partn{C}_\G$ for every $C \subseteq D$.
\end{lem}

\begin{proof}
  By Lemma \ref{lem:intpres}, the intrinsic sets of $\G' = \G[D]$ are
  precisely the intrinsic sets of $\G$ that are subsets of $D$, with
  the same associated recursive heads.  Hence the result follows from
  the definition of $\prec$.
\end{proof}

\begin{lem} \label{lem:distpart}
If $\G$ is such that $V = D_1 \dot\cup D_2$ for sets $D_1, D_2$ not 
connected by bidirected edges, then
\begin{align*}
\partn{C}_\G = \partn{C \cap D_1}_\G \cup \partn{C \cap D_2}_\G.
\end{align*}
\end{lem}

\begin{proof}
  Since every intrinsic set (and therefore recursive head) is a subset
  of either $D_1$ or $D_2$, the result follows from Propositions
  \ref{prop:part} and \ref{prop:rec-head-suitable} in the Appendix.
\end{proof}

\section{Parameterization} \label{sec:main}

We are now in a position to introduce the parameterization.  Recall
that $T$ denotes the tail associated with a recursive head $H$.  We
will present the parameterization for binary variables only, i.e.\ those with 
state-space $\X_v \equiv \{0,1\}$, each $v \in V\dot\cup W$; the extension
to non-binary discrete variables is conceptually simple but notationally
cumbersome.  Appendix \ref{sec:disc} contains notes on the general case.

\begin{dfn}
  Let $\G$ be a CADMG with random vertices $V$ and fixed vertices $W$.
  We say that $p_{V|W}$ is \emph{parameterized according to $\G$}, and
  write $p_{V|W} \in \PA(\G)$, if it can be written in the form:
\begin{align}
\label{eqn:param}
p_{V|W}(x_V \,|\, x_W) &= 
\sum_{C:O \subseteq C \subseteq V} (-1)^{|C \setminus O|} \prod_{H \in \partn{C}_\G} q_H(x_T), && x_{VW} \in \X_{VW},
\end{align}
where we define $O \equiv O(x_V) \equiv \{v \in V \,|\, x_v = 0\}$.  Here
$q_H(x_T) \in \reals$ for each $H \in \mathcal{H}$, $x_T \in \X_{T}$, 
and $T \equiv T(H)$ is the tail associated with the recursive head $H$.
\end{dfn}

Note that if $C = \emptyset$ then the product is empty, which we 
define to be equal to 1.
It will be shown in Section \ref{sec:smooth} that if $p_{V|W}$ is of
the above form then $q_H(x_T) \in [0,1]$ for all $H$ and $x_T$, or can 
be chosen to be so.  In
fact, if the graph is interpreted causally, then each $q_H(x_T)$ is
the same as $p_{H|T}(0_H \,|\, \Do(x_T))$.  

\subsection{Comparison to Other Graphical Parameterizations}

It is worth remarking on some special cases of the parameterization: 
if $\G$ is a DAG then each $H$ is a singleton
$\{h\}$, and (\ref{eqn:param}) is just the familiar parameterization
in terms of conditional probability tables using corner-point
identifiability constraints:
$q_H(x_T) = p_{h|\pa(h)}(0_{h} \,|\, x_{\pa(h)})$.  
If $\G$ has only bidirected edges then $T = \emptyset$, and 
(\ref{eqn:param}) reduces
to the parameterization given in \citet{drton:richardson:08}.  If $\G$ 
has a chain graph structure, i.e.\ the districts can be ordered so
that $v \rightarrow w$ only if $v$'s district is strictly before $w$'s, 
then the parameterization reduces
to that given in \citet{drton:09}.  

A comparison with the parameterization of \citet{evans:14} is more
subtle.  Since the ordinary Markov models in that paper only use the weaker 
requirement (i') (see Section \ref{sec:nested}) we would expect 
that they generally have a larger dimension than the nested model 
for the same graph, and therefore use a different parameterization.  
If the models are the same, and if each intrinsic set can be obtained
from a single marginalization step followed by factorization, then
the `ordinary' heads and tails will be the same as the recursive
heads and tails, and hence the parameterization will be identical. 

However, even if the ordinary and nested models 
are the same, the parameterizations can be different.  Consider the graph in 
Figure \ref{fig:16not18} (a modified version of $\mathcal{L}$).  
In this case the ordinary and nested models are the same and both 
represent the distributions for which $X_5 \indep X_3 \,|\, X_2$ and 
$X_4 \indep X_2 \,|\, X_3$; this is
the same as the corresponding maximal ancestral graph
model.  Since the set $\{2,4,5\}$ is a recursive head 
the nested parameterization includes 
the quantity $q_{245}(x_3) = P(X_2 = X_4 = X_5 = 0 \mid \Do(x_3))$ 
(see Theorem \ref{thm:recover}), whereas the ordinary 
parameterization does not have such a head, and 
uses only ordinary conditional probabilities
such as $P(X_4 = 0, X_5 = 0 \mid x_2, x_3)$. 

In general, the number of parameters in the nested model is no 
greater than the number in the ordinary Markov model, though this
number can be quite large even for sparse graphs if the districts 
are large.  The number of parameters for a particular district will 
be at least quadratic in the district size, this most parsimonious
case occurring if the district is a single chain.  The number of 
parameters may grow exponentially in the number of vertices, 
even for models with only a linear number of edges: for example 
if we have a `star' graph with all bidirected edges (this is equivalent
to a star-shaped DAG with all edges pointing to the central
node).  Such large models are potentially undesirable, and methods
to reduce the parameter count are suggested by \citet{shpitser:13}.

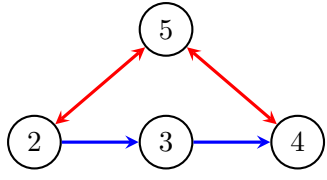
\begin{figure}
\begin{center}
\begin{tikzpicture}
[>=stealth, node distance=17.5mm]
\pgfsetarrows{latex-latex};
\begin{scope}[xshift=8cm]
 \node[rv] (2) {2};
 \node[rv, right of=2] (3) {3};
 \node[rv, right of=3] (4) {4};
 \node[rv, above of=3, yshift=-2.5mm] (5) {5};
 \draw[da] (2) -- (3);
 \draw[da] (3) -- (4);
 \draw[ba] (4) -- (5);
 \draw[ba] (2) -- (5);
\end{scope}
\end{tikzpicture}
\caption{An ADMG whose nested and ordinary Markov models are the same, but for which
the parameterizations of \citet{evans:14} and this paper are
distinct.}
\label{fig:16not18}
\end{center}
\end{figure}

\subsection{Main Results}

We will show that distributions are parameterized according to $\G$
precisely when they recursively factorize according to $\G$, so that
in fact $\RF(\G) = \PA(\G)$.  In particular, a distribution of the
form (\ref{eqn:param}) satisfies properties (i) and (ii) of the
recursive factorization.  This is shown by the next two lemmas.

\begin{lem} \label{lem:distfact} Let $\G$ be a CADMG with random
  vertices $V=D_1 \dot\cup \cdots \dot\cup D_l$, such that for
  $i \neq j$ there is no bidirected edge in $\G$ 
  from a vertex in $D_i$ to a vertex in $D_j$.  Then for all $x_{VW} \in \X_{VW}$ 
  and $O \equiv O(x_V)$,
\begin{align*}
  \sum_{O \subseteq C \subseteq V} (-1)^{|C \setminus O|} \prod_{H
    \in \partn{C}_\G} q_H(x_T) &= \prod_{i=1}^l \sum_{O_i \subseteq C \subseteq
    D_i} (-1)^{|C \setminus O_i|} \prod_{H \in \partn{C}_\G} q_H(x_T), 
\end{align*}
where $O_i = O \cap D_i$.
\end{lem}

\begin{proof}
  We prove the result for $l=2$, from which the general result follows
  by induction.  From Lemma \ref{lem:distpart}
\begin{align*}
\prod_{H \in \partn{C}_\G} q_H(x_T) = \prod_{H \in \partn{C \cap D_1}_\G} q_H(x_T) \times \prod_{H \in \partn{C \cap D_2}_\G} q_H(x_T).
\end{align*}
In addition if $C_i = C \cap D_i$, then $C \setminus O = (C_1
\setminus O_1) \cup (C_2 \setminus O_2)$ and this is the union of two
disjoint sets, so $|C \setminus O| = |C_1 \setminus O_1|+|C_2
\setminus O_2|$.  Hence
\begin{align*}
  \sum_{O \subseteq C \subseteq V} (-1)^{|C \setminus O|} \prod_{H \in \partn{C}_\G} q_H(x_T) 
  &=   \sum_{O \subseteq C \subseteq D_1 \cup D_2} (-1)^{|C \setminus O|} \prod_{H \in \partn{C \cap D_1}_\G} q_H(x_T) \prod_{H \in \partn{C \cap D_2}_\G} q_H(x_T)\\
  &=   \sum_{O_1 \subseteq C_1 \subseteq D_1} (-1)^{|C_1 \setminus O_1|} \prod_{H \in \partn{C_1}_\G} q_H(x_T)\\ 
  & \qquad \times \sum_{O_2 \subseteq C_2 \subseteq D_2} (-1)^{|C_2 \setminus O_2|} \prod_{H \in \partn{C_2}_\G} q_H(x_T).
\end{align*}
\end{proof}

\begin{lem} \label{lem:marfact} Let $\G$ be a CADMG with a 
random vertex $v$.
Then for all $x_{VW} \in \X_{VW}$ and $O \equiv O(x_V)$,
\begin{align*}
\lefteqn{\sum_{O \subseteq C \subseteq V} (-1)^{|C \setminus O|} \prod_{H
    \in \partn{C}_\G} q_H(x_T)}\\
&=  
      \sum_{O \subseteq C \subseteq V \setminus \{v\}} (-1)^{|C \setminus O|} \prod_{H
    \in \partn{C}_\G} q_H(x_T) 
    -\sum_{O \cup \{v\} \subseteq C \subseteq V} (-1)^{|C \setminus (O \cup \{v\})|} \prod_{H
    \in \partn{C}_\G} q_H(x_T).
\end{align*}
\end{lem}

\begin{proof}
Separating the sum into those subsets $C$ that contain $v$ and those which do 
not gives
\begin{align*}
\lefteqn{\sum_{O \subseteq C \subseteq V} (-1)^{|C \setminus O|} \prod_{H
    \in \partn{C}_\G} q_H(x_T)}\\
&=  
      \sum_{O \subseteq C \subseteq V \setminus \{v\}} (-1)^{|C \setminus O|} \prod_{H
    \in \partn{C}_\G} q_H(x_T) 
    + \sum_{O \cup \{v\} \subseteq C \subseteq V} (-1)^{|C \setminus O|} \prod_{H
    \in \partn{C}_\G} q_H(x_T),
\end{align*}
which is seen to be the same as the given expression by including a factor of
$-1$ inside and outside the second sum. 
\end{proof}

We now move to the main result of the paper.

\begin{thm} \label{thm:main}
The kernel $p_{V|W}$ recursively factorizes according to $\G$ if and only if it
  is parameterized according to $\G$.
\end{thm}

\begin{proof}
  Throughout the proof we will write the partitions of vertices in a
  CADMG as $\partn{\cdot}_\G$ regardless of which graph we are dealing
  with; since all the graphs we consider are reachable from $\G$, this
  is justified by Lemma \ref{lem:headpreserve}.

  We proceed by induction on the size of $V$.  If $V = \{v\}$ then
  recursive factorization is by definition, so the condition
  holds for any distribution.  On the other hand, parameterization
  entails
  \begin{align}
  p_{v|W}(0_v \,|\, x_W) &= q_v(x_{\pa(v)}),
  &p_{v|W}(1_v \,|\, x_W) &= 1 - q_v(x_{\pa(v)}), \label{eqn:elem}
  \end{align}
  which follows from setting $q_v(x_{\pa(v)}) = p_{v|W}(0_v \,|\,
  x_W)$ and the fact that $p_{v|W}(0_v \,|\, x_W) + p_{v|W}(1_v \,|\, x_W) = 1$
  because $p_{v|W}$ is a probability distribution; hence parameterization
  also holds for any distribution with one random variable.

\medskip

($\Leftarrow$)  Now consider a general $V$ and suppose $p_{V|W}$ is parameterized
  according to $\G$.  If $\G$ has multiple districts then, by Lemma
  \ref{lem:distfact}, the kernel factorizes into pieces which are parameterized
  according to $\G[D_i]$, and so by the
  induction hypothesis recursively factorize according to $\G[D_i]$.

  Otherwise take any $a \in \sterile_\G(V)$, and fix $x_{VW\setminus \{a\}} \in \X_{VW \setminus \{a\}}$; 
  let $O = \{v \in V \setminus \{a\} \,|\, x_v = 0\}$, so then
\begin{align*}
\sum_{x_a} p(x_V \,|\, x_W) &= p(x_{V\setminus a}, 0_a \,|\, x_W) + p(x_{V\setminus a}, 1_a \,|\, x_W)\\
&=\sum_{O \cup \{a\} \subseteq C \subseteq V} (-1)^{|C \setminus (O \cup \{a\})|} \prod_{H \in \partn{C}_\G} q_H(x_T) 
+ \sum_{O \subseteq C \subseteq V} (-1)^{|C \setminus O|} \prod_{H \in \partn{C}_\G} q_H(x_T)\\
&= \sum_{O\subseteq C \subseteq V \setminus \{a\}} (-1)^{|C \setminus O|} \prod_{H \in \partn{C}_{\G}} q_H(x_T)
\end{align*}
by Lemma \ref{lem:marfact}.
By the induction hypothesis this last expression recursively factorizes according
to $\G[V \setminus \{a\}] = \fm_{V \setminus a}(\G)$, and this
extends easily to any random-ancestral margin $V \setminus B$ by sequentially
marginalizing the variables in $B$.  Hence
$p_{V|W}$ obeys properties (i) and (ii) of recursive factorization,
and therefore recursively factorizes according to $\G$.

\medskip

($\Rightarrow$) Conversely, suppose that $p_{V|W}$ recursively factorizes according to
$\G$.  In this direction we will strengthen the induction hypothesis
slightly and show that if $p_{V|W}$ recursively factorizes according
to $\G$ then $p_{V|W}$ is parameterized according to $\G$, \emph{and}
that for each parameter $q_H(x_T)$, either
$p_{T\setminus W|W}(x_{T\setminus W} \,|\, x_{T \cap W}, y_{W
  \setminus T}) > 0$
for some $y_{W \setminus T}$, in which case $q_H(x_T)$ is uniquely
recoverable from $p_{V|W}$; or
$p_{T\setminus W|W}(x_{T\setminus W} \,|\, x_{T \cap W}, y_{W
  \setminus T}) = 0$
for all $y_{W \setminus T}$, in which case $q_H(x_T)$ can take any
value.
For the base case with $|V|=1$ the result follows from the
derivation in (\ref{eqn:elem}).

If $\G$ has multiple districts then, by definition, $p_{V|W}$
factorizes into pieces which themselves recursively factorize
according to the districts $\G[D_i]$, and by the induction hypothesis
each factor is parameterized according to $\G[D_i]$.  Applying Lemma
\ref{lem:distfact} it follows that $p_{V|W}$ is parameterized
according to $\G$, and no parameters are shared between these factors
by Lemma \ref{lem:distpart}.

For uniqueness of $q_H(x_T)$, note that this parameter only
appears in the expansion for probabilities $p_{V|W}(x_V \,|\, x_W)$ (i.e.\ those
indexed by the same values $x_T$).  
If $p_{T \setminus W|W}(x_{T\setminus W}\,|\, x_W) > 0$ 
then the factorization of $p_{V|W}$ is
unique for these values by Proposition \ref{prop:gform}, and each
factor is also positive for those values of $x_T$ by Lemma
\ref{lem:pos}; thus $q_H(x_T)$ is uniquely recoverable from that
factor by the strengthened induction hypothesis.

If $p_{T\setminus W|W}(x_{T\setminus W} \,|\, x_W) = 0$ then by Lemma
\ref{lem:pos} there is some $t \in T \setminus W$ and $x_{V \setminus T}$ such
that every version of
$p_{t | \pre_<(t),W}(x_t \,|\, x_{\pre_<(t)}, x_W) = 0$.
We split into two cases: either $t$ is in the same district as $H$, or
not; let $D_1$ be the district containing $H$, and the associated kernel $r_1$.  If $t$
is in $D_1$  then it follows from Proposition
\ref{prop:gform} that
$r_1(x_{T \cap D_1} \,|\, x_{T \setminus D_1}) = 0$, and so by the
induction hypothesis applied to $\G[D_1]$ we get that $q_H(x_T)$ can
take any value.  Otherwise if $t$ is in a different district (say
$D_2$) then it follows from Proposition \ref{prop:gform} that
$r_2(x_{T \cap D_2} \,|\, x_{\pa(D_2) \setminus D_2}) = 0$; so
clearly whatever the value of any other factor, including $r_1$, 
the product will always be zero.

%
%
%

\medskip

Now suppose $\G$ has a single district $V$; it follows from the 
definitions that $V$ is intrinsic with recursive head 
$H^* = \sterile_\G(V)$ and tail $T^* = (V \cup W) \setminus H^*$.  
For any vertex $h \in H^*$ the set $V \setminus \{h\}$ is
random-ancestral, so the margin $p_{V\setminus h|W}$ recursively
factorizes with respect to $\G[V \setminus \{h\}]$, and therefore (by the induction
hypothesis) is also parameterized according to $\G[V \setminus \{h\}]$.  Every recursive head
$H$ other than $H^*$ is found in at least one random-ancestral margin
$V \setminus \{h\}$ of $\G$, so applying the induction hypothesis to
$\G[V \setminus \{h\}]$ we obtain either a well defined parameter, or
determine that its value is irrelevant.

If two or more random-ancestral margins contain the recursive head $H$, note
that by Corollary \ref{cor:anccap} there is a `smallest' such margin
$p_{\an(H) \setminus W|W}$ containing $H$; all other random-ancestral
margins must be consistent with this margin, and therefore by the induction
hypothesis they will agree either on a value for $q_H(x_T)$ or agree
that it is arbitrary.  So for every random-ancestral set $A \subsetneq V$ the 
margin $\G[A]$ is parameterized according to $p_{A|W}$ and any 
parameters that two or more of these margins jointly use either are consistent, 
or can be chosen to be consistent.
 
The only recursive head not found in a random-ancestral margin is $H^*$, so the
only parameter yet to be defined is $q_{H^*}(x_{T^*})$.  We define this
to be any version of $p_{H^*|T^*}(0_{H^*} \,|\, x_{T^*})$; this is
well defined if $p(x_{T^* \setminus W} \,|\, x_{T^* \cap W}) > 0$, and
arbitrary otherwise.  Then
\begin{align*}
p_{V|W}(0_{H^*}, x_{V \setminus H^*} \,|\, x_W) &= q_{H^*}(x_{T^*}) \cdot p(x_{V \setminus H^*} \,|\, x_W).
\end{align*}
Since $V \setminus H^*$ is a random-ancestral margin of $\G$, 
it follows that $p(x_{V \setminus H^*} \,|\, x_W)$
is parameterized according to $\G[V \setminus H^*]$, and so
\begin{align*}
p_{V|W}(0_{H^*}, x_{V \setminus H^*} \,|\, x_W) &= p(0_{H^*} \,|\, x_{V \setminus H^*}, x_W) \cdot \prod_{O \subseteq C \subseteq V \setminus H^*} (-1)^{|C \setminus O|} \prod_{H \in \partn{C}_{\G}} q_H(x_T)\\
&= \prod_{O \subseteq C \subseteq V} (-1)^{|C \setminus O|} \prod_{H \in \partn{C}_{\G}} q_H(x_T).
\end{align*}
This gives the required result if $x_{h} = 0$ for all $h \in H^*$.  
On the other hand, if $x_h = 1_h$ for some $h \in H^*$ then using
a second induction on the number of zeros in $x_{H^*}$ we have
\begin{align*}
p(x_{V\setminus h}, 1_h \,|\, x_W) &= p(x_{V\setminus h} \,|\, x_W) - p(x_{V\setminus h}, 0_h \,|\, x_W)\\
&= \sum_{O\subseteq C \subseteq V \setminus \{h\}} \!\!\!\! (-1)^{|C \setminus O|} \!\! \prod_{H \in \partn{C}_{\G}} q_H(x_T)
 - \sum_{O \cup \{h\} \subseteq C \subseteq V} \!\!\!\! (-1)^{|C \setminus (O \cup \{h\})|} \!\! \prod_{H \in \partn{C}_\G} q_H(x_T) \\
&= \sum_{O \subseteq C \subseteq V} (-1)^{|C \setminus O|} \prod_{H \in \partn{C}_\G} q_H(x_T)
\end{align*}
using Lemma \ref{lem:marfact}.  
Hence every probability $p_{V|W}(x_V \,|\, x_W)$ is of the required form.
\end{proof}

\subsection{Model Smoothness} \label{sec:smooth}

For some ADMGs $\G$, the parameters $q_H(x_T)$
are just (versions of) the ordinary conditional probabilities
$P(X_H = 0 \,|\, X_T = x_T)$, and hence the alternating sum is 
similar to the M\"obius form of the parameterization studied in
\citet{evans:14} in the context of `ordinary' Markov models.  
However we have already seen that not all of the
parameters can be interpreted this way; recall the example in Section
\ref{sec:nested} for Figure \ref{fig:verma}(b).  In this case, as
noted in Example \ref{exm:order},
$q_{EY}(x, m) = r_{EY}(0,0\,|\, x,m)$ is not an ordinary conditional
probability, but if the graph is interpreted causally then it is the
conditional probability of $\{E = Y = 0\}$ after intervening to fix
$\{X = x, M = m\}$: 
\begin{align*}
  q_{EY}(x, m) &= p_{E|X}(0 \,|\, x) \cdot p_{Y|XME}(0 \,|\, x, m, 0)\\
               &= P(Y=E=0 \,|\, \Do(X=x, M=m)).  
\end{align*}
By the requirement that the graph is `interpreted causally' we mean
that it is the latent projection of a causal DAG in the sense of
\citet[][Definition 1.3.1]{pearl:09}.  This result holds more
generally.


\begin{thm} \label{thm:recover} If $p_{V|W}$ is strictly positive and
  recursively factorizes according to some CADMG $\G$, then all the
  parameters $q_H(x_T)$ are unique and can be smoothly recovered from
  $p_{V|W}$ (i.e.\ there is an infinitely differentiable function 
  from $p_{V|W}$ to the $q_H(x_T)$).

  In addition, if the graph is interpreted causally then 
\begin{align*}
q_H(x_T) = P(X_H = 0_H \,|\, \Do(X_T = x_T)).
\end{align*}
\end{thm}

\begin{proof}
  The first claim follows directly from the proof of Theorem
  \ref{thm:main}, since the operations involved are just summations
  and divisions by positive quantities; 
  the fact that $p_{V|W}$ is strictly positive ensures
  that each parameter is always uniquely defined rather than being
  arbitrary.

  The second part follows from the algorithm in \citet{tian:02}.
\end{proof}

We remark that if the distribution is not strictly positive then it follows from 
the `$\Rightarrow$' part of the proof of Theorem \ref{thm:main} that the
parameters $q_H(x_T)$ are uniquely defined if and only if 
$p(x_{V \cap T} \,|\, x_{W \cap T}, y_{W \setminus T}) > 0$ for some
$y_{W \setminus T}$.  In the case that $W = \emptyset$ and $\G$ 
is an ADMG, this reduces to $q_H(x_T)$ being uniquely defined 
if and only if $p(x_T) > 0$.

We now return to the generality of a finite discrete state-space $\X_v$ for 
each $X_v$.  Let $\tilde\X_v$ be the same set with some arbitrary
entry removed (so that $|\tilde\X_v| = |\X_v|-1$).  Then for any set
$C$ let $\tilde\X_C \equiv \times_{v\in C}\, \tilde\X_v$.

\begin{cor}
The set of strictly positive distributions obeying the recursive factorization 
property with respect to a CADMG $\G$ is a curved exponential family of 
dimension
\begin{align*}
d(\G) = \sum_{H \in \mathcal{H}(\G)} |\tilde\X_H| \cdot |\X_T|.
\end{align*}
\end{cor}

\begin{proof}
  Theorem \ref{thm:recover} shows that there is a smooth (infinitely
  differentiable) map from positive distributions obeying the
  recursive factorization to the model parameters; it is clear from
  the form of the parameterization that the map from parameters to the
  probabilities is also smooth.  The result follows by the same
  argument as Theorem 6.5 of \citet{evans:14}.
\end{proof}

This result allows us to invoke standard statistical theory within
this class of models.  For example, if $\G'$ is a subgraph of $\G$,
then we can perform a hypothesis test of $H_0: p_{V|W} \in \RF(\G')$
versus $H_1: p_{V|W} \in \RF(\G)$ by comparing the likelihood ratio
statistic to a $\chi^2_k$ distribution, where $k = d(\G) - d(\G')$. 

Fitting these models is relatively straightforward given the explicit
maps between parameters and probabilities, and maximum likelihood 
estimation can be performed using the same method as in \citet{evans:10}. 
The parameters $q_H(x_T)$ are clearly variation dependent, which 
can cause algorithmic complications and interpretability problems.  
A log-linear parameterization of the kind given in \citet{evans:13} 
can relatively easily be adapted to nested models; see also 
\citet{shpitser:13}.

\section{Examples} \label{sec:exm}

The Wisconsin Longitudinal Study \citep{wls} is a panel study of over 10,000
people who graduated from Wisconsin High Schools in 1957.  We consider
males who, when asked in 1975, had either been drafted or had not
served in the military at all; after removing missing data this left
1,676 respondents.  We wish to know whether, after controlling for
family income and education, being drafted had a significant effect on
future earnings.

The variables measured were:
\begin{itemize}
\item $X$, an indicator of whether family income in 1957 was above
  \$5k;
\item $Y$, an indicator of whether the respondent's income in 1992 was above \$37k;
\item $M$, an indicator of whether the respondent was drafted into the military;
\item $E$, an indicator of whether the respondent had education beyond high school.
\end{itemize}

Dichotomizations for $X$, $Y$ and $E$ were chosen to be close to the
median values of the original variables.  The data are shown in Table
\ref{tab:wls}; in each case the value 1 corresponds to the statement
above being true, 0 otherwise.  One possible model is that future
income is unrelated to family income at the time of graduation after
controlling for military service and level of education.  This
suggests the graph in Figure \ref{fig:mil1}(a), where the directed
edge from $X$ to $Y$ is not present.  We can fit this model using the
parameterization and an algorithm based on the one given by
\citet{evans:10}; the resulting fit has a deviance of 31.3 on 2
degrees of freedom, strongly suggesting that the model should be
rejected.  Unsurprisingly, the graph in Figure \ref{fig:verma}(b) is
also rejected for these data.

\begin{figure}
\begin{center}
\begin{tikzpicture}
[>=stealth, node distance=15mm]
\pgfsetarrows{latex-latex};
\begin{scope}
 \node[rv] (1) {$X$};
 \node[rv, right of=1] (2) {$E$};
 \node[rv, right of=2] (3) {$M$};
 \node[rv, right of=3] (4) {$Y$};
 \draw[da] (1) -- (2);
 \draw[da] (2) -- (3);
 \draw[da] (3) -- (4);
\draw[ba] (2.40) .. controls +(40:1) and +(140:1) .. (4.140);
\draw[da] (1.320) .. controls +(320:1) and +(220:1) .. (3.220);
\node[below of=2, xshift=7.5mm, yshift=2mm] {(a)};
\end{scope}
\begin{scope}[xshift=7cm]
 \node[rv] (1) {$X$};
 \node[rv, right of=1] (2) {$E$};
 \node[rv, right of=2] (3) {$M$};
 \node[rv, right of=3] (4) {$Y$};
 \draw[da] (1) -- (2);
 \draw[da] (2) -- (3);
\draw[ba] (2.40) .. controls +(40:1) and +(140:1) .. (4.140);
\draw[da] (1.320) .. controls +(320:1) and +(220:1) .. (4.220);
\node[below of=2, xshift=7.5mm, yshift=2mm] {(b)};
\end{scope}
\end{tikzpicture}
\caption{Two models for the Wisconsin miltary service data. (a) A
  proposed but rejected model; (b) a well-fitting model.  See text for
  discussion.}
\label{fig:mil1}
\end{center}
\end{figure}
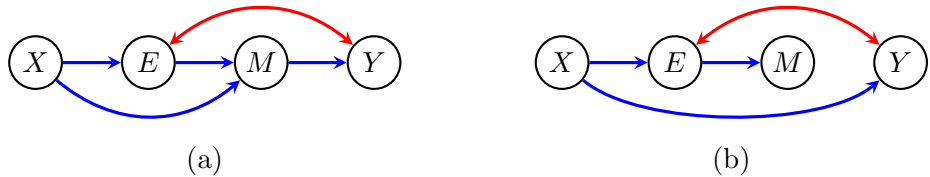

On the other hand the model shown in Figure \ref{fig:mil1}(b) has a
deviance of 5.57 on 6 degrees of freedom, which indicates a good fit.
Note that this implies that there is no evidence of a significant
effect of being drafted on future income, even though marginally there
is a strong negative correlation.  Models obtained by removing any
additional edges are strongly rejected.  Under this model the
probability of having a high income in 1992 is estimated as 0.50
(standard error 0.018) if the family had high income, and 0.36 (0.016)
if not.

In other words, we estimate
\begin{align*}
P(Y=1 \,|\, \Do(X=1)) &= 0.50 &
P(Y=1 \,|\, \Do(X=0)) &= 0.36,
\end{align*}
indicating a strong causal effect.

\begin{table}
\begin{center}
\begin{tabular}{|c|cc|c|c|cc|}
\multicolumn{3}{c}{$X=0,E=0$}&\multicolumn{1}{c}{}&\multicolumn{3}{c}{$X=1,E=0$}\\
\cline{1-3}\cline{5-7}
$M \backslash Y$&0&1&&$M \backslash Y$&0&1\\
\cline{1-3}\cline{5-7}
0&241&162&&0&161&148\\
1&53&39&&1&33&29\\
\cline{1-3}\cline{5-7}
\multicolumn{7}{c}{}\\
\multicolumn{3}{c}{$X=0,E=1$}&\multicolumn{1}{c}{}&\multicolumn{3}{c}{$X=1,E=1$}\\
\cline{1-3}\cline{5-7}
$M \backslash Y$&0&1&&$M \backslash Y$&0&1\\
\cline{1-3}\cline{5-7}
0&82&176&&0&113&364\\
1&13&16&&1&16&30\\
\cline{1-3}\cline{5-7}
\end{tabular}
\end{center}
\caption{Data from the Wisconsin Longitudinal Study.}
\label{tab:wls}
\end{table}

\subsubsection*{Acknowledgements}

This research uses data from the Wisconsin Longitudinal Study (WLS) of
the University of Wisconsin-Madison, which is supported principally by
the National Institute on Aging.  This research was supported by a SQuaRE 
grant from the American Institute of Mathematics. 
We thank two anonymous reviewers and
an associate editor for suggesting several improvements to the paper.

\bibliographystyle{plainnat}
\bibliography{mybib}

\appendix

\section{Proof of the Verma Constraint} \label{sec:dorm}

Note that 
  \begin{align*}
\sum_{e} p(e \,|\, x) \cdot p(y \,|\, x, m, e) &= \sum_{e} \frac{p(x, m, e, y)}{p(x) \cdot p(m \mid x, e)}\\
&= \sum_{e} \frac{\sum_u p(u, x, m, e, y)}{p(x) \cdot p(m \mid x, e)}
\intertext{by elementary laws of conditional probability.  Applying the 
usual factorization of the DAG in Figure \ref{fig:verma}(a), we obtain}
&= \sum_{e} \frac{\sum_u p(u) \cdot p(x) \cdot p(e \mid x, u) \cdot p(m \mid e) \cdot p(y \mid m, u)}{p(x) \cdot p(m \mid x, e)}\\
\intertext{noting that $M \indep X \mid E$, and cancelling, gives}
&= \sum_{e, u} p(u)  \cdot p(e \mid x, u) \cdot p(y \mid m, u)\\
&= \sum_{u} p(u)  \cdot p(y \mid m, u),
\end{align*}
which does not depend upon $x$.

\section{Partitions}

Let $V$ be an arbitrary finite set, and let $\mathcal{H}$ be an
arbitrary collection of non-empty subsets of $V$, with the restriction
that $\{v\} \in \mathcal{H}$ for all $v \in V$ (i.e.\ all singletons
are in $\mathcal{H}$). A partial ordering $\prec$ on the elements of
$\mathcal{H}$ will be said to be {\em partition suitable} if for any
$H_1, H_2 \in \mathcal{H}$ with $H_1 \cap H_2 \neq \emptyset$, there
exists $H^* \in \mathcal{H}$ such that $H^* \subseteq H_1 \cup H_2$
and $H_i \preceq H^*$ for each $i=1,2$.  (Here $H_1 \preceq H_2$ means
$H_1 \prec H_2$ or $H_1 = H_2$.)  

Define a function $\Phi$ on subsets of $V$ such that $\Phi(W)$ `picks
out' the set of $\prec$-maximal elements of $\mathcal{H}$ that are
subsets of $W$.  That is,
\begin{align*}
\Phi(W) \equiv \{H \in \mathcal{H} \,|\, H \subseteq W \text{ and } H \nprec H' \text{ for all other } H' \subseteq W\}.
\end{align*}
Define
\begin{align*}
\psi(W) \equiv W \setminus \bigcup_{C \in \Phi(W)} C.
\end{align*}
Now recursively define a function $[\cdot]$ that partitions subsets of 
$V$: define $[\emptyset] = \emptyset$, and 
\begin{align*}
[W] \equiv \Phi(W) \cup [\psi(W)].
\end{align*}
It is clear that $\cup_{A \in [W]} A = W$.  

The next proposition shows that $[W]$ is indeed a partition of $W$.

\begin{prop} \label{prop:disjoint}
If $H_1, H_2 \in \Phi(W)$ with $H_1 \neq H_2$ then $H_1 \cap H_2 =\emptyset$.
\end{prop}

\begin{proof}
  Suppose $H_1 \cap H_2 \neq \emptyset$.  Then by partition
  suitability, there exists $H^* \subseteq H_1 \cup H_2$ with $H^*
  \succeq H_1, H_2$, and in particular $H^* \succ H_i$ for at least
  one of $i=1,2$.  Hence at least one of the $H_i$ is not maximal in
  $W$.
\end{proof}

\begin{prop} \label{prop:stillmax}
  If $A \subseteq W_1 \subseteq W_2$, and $A \in \Phi(W_2)$ then $A
  \in \Phi(W_1)$.
\end{prop}

\begin{proof}
  If $A$ is maximal amongst those recursive heads that are subsets of $W_2$,
  then it is certainly still maximal amongst those that are subsets
  of $W_1$, since there are fewer such recursive heads.
\end{proof}



\begin{prop} \label{prop:rmv}
If $C \in [W]$, then $[W] = \{C\} \cup [W \setminus C]$.
\end{prop}

\begin{proof}
  We proceed by induction on the size of $W$.  If $[W] = \{C\}$,
  including any case in which $|W|=1$, the result is trivial.

If $C$ is not maximal with respect to $\prec$ in $W$, then $\Phi(W) = \Phi(W \setminus C)$, and so
\begin{align*}
[W] &= \Phi(W) \cup [\psi(W)]\\
  &= \Phi(W \setminus C) \cup [\psi(W)],
\end{align*}
and the problem reduces to showing that $[\psi(W)] = \{C\} \cup
[\psi(W \setminus C)]$, which follows from the induction hypothesis.
Thus, suppose $C \in \Phi(W)$.

Now by Proposition \ref{prop:stillmax}, $\Phi(W \setminus C) \cup
\{C\} \supseteq \Phi(W)$, and if equality holds we are done.
Otherwise let $C_1, \ldots, C_k$ be the sets in $\Phi(W \setminus C)$
but not in $\Phi(W)$.  These sets are maximal in $W \setminus C$, so
they are in $\Phi(\psi(W))$ by Proposition \ref{prop:stillmax}, since
by hypothesis, $\psi(W) \subseteq W\setminus C$.  Then the problem
reduces to showing that
\begin{align*} 
[\psi(W)] = \{C_1, \ldots, C_k\} \cup [\psi(W) \setminus
(C_1 \cup \cdots \cup C_k)], 
\end{align*} 
which follows from repeated
application of the induction hypothesis.
\end{proof}

\begin{prop} \label{prop:part} Let $D_1, \ldots, D_k$ be a partition
  of $V$, and suppose that each $H \in \mathcal{H}$ is contained
  within some $D_i$.  Let $\prec$ be a partition-suitable partial
  ordering.  Then
\begin{align*}
[W] = \bigcup_{i=1}^k [W \cap D_i].
\end{align*}
\end{prop}

\begin{proof}
We prove the case $k=2$, from which the general result follows by repeated applications.  If either $W \cap D_1$ or $W \cap D_2$ are empty, then the result is trivial.  By definitions
\begin{align*}
[W] = \Phi(W) \cup [\psi(W)];
\end{align*}
$\psi(W)$ is strictly smaller than $W$, so by the induction hypothesis
\begin{align*}
[W] = \Phi(W) \cup [\psi(W) \cap D_1] \cup [\psi(W) \cap D_2].
\end{align*}
By hypothesis $\Phi(W) = \mathcal{C}_1 \cup \mathcal{C}_2$ where each
$H \in \mathcal{C}_i$ is a subset of $D_i$; since the elements of
$\mathcal{C}_i$ are maximal with respect to $\prec$ in $W$, they are
also maximal in $W \cap D_i$.  Hence $\mathcal{C}_i \subseteq \Phi(W
\cap D_i)$, and then applying Proposition \ref{prop:rmv} gives
\begin{align*}
\mathcal{C}_i \cup [\psi(W) \cap D_i] 
&= [W \cap D_i],
\end{align*}
because $(\psi(W) \cap D_i) \cup \bigcup \mathcal{C}_i = W \cap D_i$.
Hence the result.
\end{proof}


\subsection{Partition Suitability of Recursive Head Ordering}

The next result, together with the previous one, shows that the
partition defined in Section \ref{sec:part} for recursive heads is
indeed a partition.

\begin{prop}\label{prop:rec-head-suitable}
$\prec$ is partition suitable for $\mathcal{H}(\G)$.  
\end{prop}

\begin{proof}
  Lemma \ref{lem:singlehead} shows that $\mathcal{H}$ contains the
  singleton vertices.  Now suppose we have two recursive heads
  $H_1,H_2$ with $H_1 \cap H_2 \neq \emptyset$.  Let the associated
  intrinsic sets be $S_1, S_2$.  Since $S_1,S_2$ are bidirected
  connected sets and they share a common element, $S_1 \cup S_2$ is
  also bidirected-connected.  Let $S^*$ be the intrinsic closure of 
  $S_1 \cup S_2$, with recursive head $H^*$.  Then $S^*$ contains both $S_1$ and $S_2$, 
  and therefore $H^* \succeq H_1, H_2$. 
  
  By Proposition \ref{prop:conn}, $H^* = \sterile_\G(S^*) \subseteq S_1 \cup S_2$;
  By definition of a recursive head, any $v \in S_1$ is either in $H_1$ or is a 
  parent of some other element of $S_1$ (and the same for $S_2$).  
  Hence $H^* \subseteq H_1 \cup H_2$.
\end{proof}

\section{General Discrete State-space} \label{sec:disc}

Lemmas \ref{lem:distfact} and \ref{lem:marfact} and Theorem 
\ref{thm:main} are stated
and proved for binary variables to avoid cumbersome notation; 
here we provide some notes on how one would adapt them to the general
case.  

Suppose that $\X_{VW}$ is possibly non-binary.  For each $v\in V$ 
pick an arbitrary element $k_v \in \X_v$ to be a corner-point.
Let $\tilde\X_v \equiv \X_v \setminus \{k_v\}$ and   
$\tilde\X_{C} \equiv \times_{v \in C}\, \tilde\X_v$.
In the binary case
we took $k_v = 1$, so that $\tilde\X_v = \{0\}$ for each $v$.

The parameters then become $q_H(x_H \mid x_T)$ for $H \in \mathcal{H}(\G)$, 
$x_H \in \tilde{\X}_H$ and $x_T \in \X_T$.  
The parameterization in (\ref{eqn:param}) becomes:
\begin{align*}
p_{V|W}(x_V \,|\, x_W) = 
\sum_{O \subseteq C \subseteq V} (-1)^{|C \setminus O|} \sum_{y_C \in \tilde{\X}_C: y_{O} = x_O} \prod_{H \in \partn{C}_\G} q_H(y_H \mid x_T),
\end{align*}
where $O \equiv O(x_V) = \{v \mid x_v \in \tilde{\X}_v\}$.  
Note that, in the binary case, the inner sum only ever has one term.

Lemma \ref{lem:distfact} goes through as before by splitting the inner 
sum up as
\[
\sum_{y_C \in \tilde{\X}_C: y_{O} = x_O} = \sum_{y_{C_1} \in \tilde{\X}_{C_1}: y_{O_1} = x_{O_1}} \sum_{y_{C_2} \in \tilde{\X}_{C_2}: y_{O_2} = x_{O_2}}.
\]
The proof of Theorem \ref{thm:main} is also the same, except that
instead of $x_h = 0$ and $x_h = 1$ the important cases become 
$x_h \in \tilde\X_h$ and $x_h = k_h$.

\end{document}